\documentclass[reqno]{amsart}
\usepackage{amsmath, amssymb, amsthm}
\usepackage{quiver} % Diagrams
\usepackage{hyperref}
\usepackage[nameinlink,capitalise,noabbrev]{cleveref}
\usepackage{stmaryrd} % for boxslash -- not sure if this is the best option
\usepackage{mathtools}
\usepackage{enumitem}
\usepackage{ifthen}
\usepackage{listings}
\usepackage{manfnt}
\usepackage{hyperref}
\hypersetup{
    colorlinks,
    linkcolor=nice-pink,
    citecolor = nice-purple
}

\usepackage{todonotes}
\usepackage[letterpaper,margin=1in]{geometry}
\usepackage[export]{adjustbox}
\usepackage{caption}
\usepackage[parfill]{parskip}
\usepackage{microtype}

\newtheoremstyle{warnsign}
  {\topsep}                % ABOVESPACE
  {\topsep}                % BELOWSPACE
  {}                       % BODYFONT
  {}                       % INDENT (empty value is the same as 0pt)
  {\bfseries}              % HEADFONT
  {. \rm\textdbend}        % HEADPUNCT
  {5pt plus 1pt minus 1pt} % HEADSPACE
  {Warning \thewarning}    % CUSTOM-HEAD-SPEC

\theoremstyle{plain}
\newtheorem{theorem}{Theorem}[section]
\newtheorem{letterthm}{Theorem}

\newtheorem{proposition}[theorem]{Proposition}
\newtheorem{corollary}[theorem]{Corollary}
\newtheorem{lemma}[theorem]{Lemma}
\newtheorem{conjecture}[theorem]{Conjecture}
\theoremstyle{definition}
\newtheorem{definition}[theorem]{Definition}
\newtheorem{example}[theorem]{Example}
\newtheorem{summary}[theorem]{Summary}
\newtheorem{remark}[theorem]{Remark}
\theoremstyle{warnsign}
\newtheorem{warning}[theorem]{Warning}
  
\title{Formal Concept Analysis and Homotopical Combinatorics}
\author[S.\ Balchin]{Scott Balchin}
\address{Mathematical Sciences Research Centre, Queen's University Belfast, UK}
\email{s.balchin@qub.ac.uk}
 
\author[B.\ Spitz]{Ben Spitz}
\address{Department of Mathematics, University of Indiana, Bloomington, IN, USA}
\email{bespitz@iu.edu}


\newcommand{\conlat}{\underline{\mathfrak{B}}}
\newcommand{\Sub}{\mathsf{Sub}}
\newcommand{\Tr}{\mathsf{Tr}}
\newcommand{\coTr}{\mathsf{coTr}}
\newcommand{\Sat}{\mathsf{Sat}}
\newcommand{\coSat}{\mathsf{coSat}}
\newcommand{\Cov}{\mathsf{Cov}}

\newcommand{\qbinom}[3]{%
  \IfNoValueTF{#3}{\binom{#1}{#2}}{%
    \mathchoice
    {\binom{#1}{#2}_{\!\!#3}}
    {\binom{#1}{#2}_{\!#3}}
    {\binom{#1}{#2}_{\!#3}}
    {\binom{#1}{#2}_{\!#3}}%
  }%
}

\DeclarePairedDelimiter\ceil{\lceil}{\rceil}
\DeclarePairedDelimiter\floor{\lfloor}{\rfloor}
\DeclarePairedDelimiter\abs{\lvert}{\rvert}

\usepackage{wrapfig}

\definecolor{nice-purple}{HTML}{9B4F96}
\definecolor{nice-pink}{HTML}{D60270}

\begin{document}

\begin{abstract}
    Formal Concept Analysis makes the fundamental observation that any finite lattice $(L, \leq)$ is determined up to isomorphism by the restriction of the relation ${\leq} \subseteq L \times L$ to the set $J(L) \times M(L)$, where $J(L)$ is the set of join-irreducible elements of $L$, and $M(L)$ is the set of meet-irreducible elements of $L$.
    
    For any finite lattice $L$ equipped with the action of a finite group $G$, we explicitly describe this restricted relation for the lattice of transfer systems $\Tr(L)$ in terms of $L$ only. We apply this to give new computations of the number of transfer systems for certain finite groups, and to produce bounds on the number of transfer systems on certain families of abelian finite groups. We also provide computer code to enable other researchers' use of these techniques.
\end{abstract}

\maketitle

\vspace{10mm}
\begin{figure}[h]
    \centering{}
    \includegraphics[scale=0.45, cframe=nice-pink 2pt]{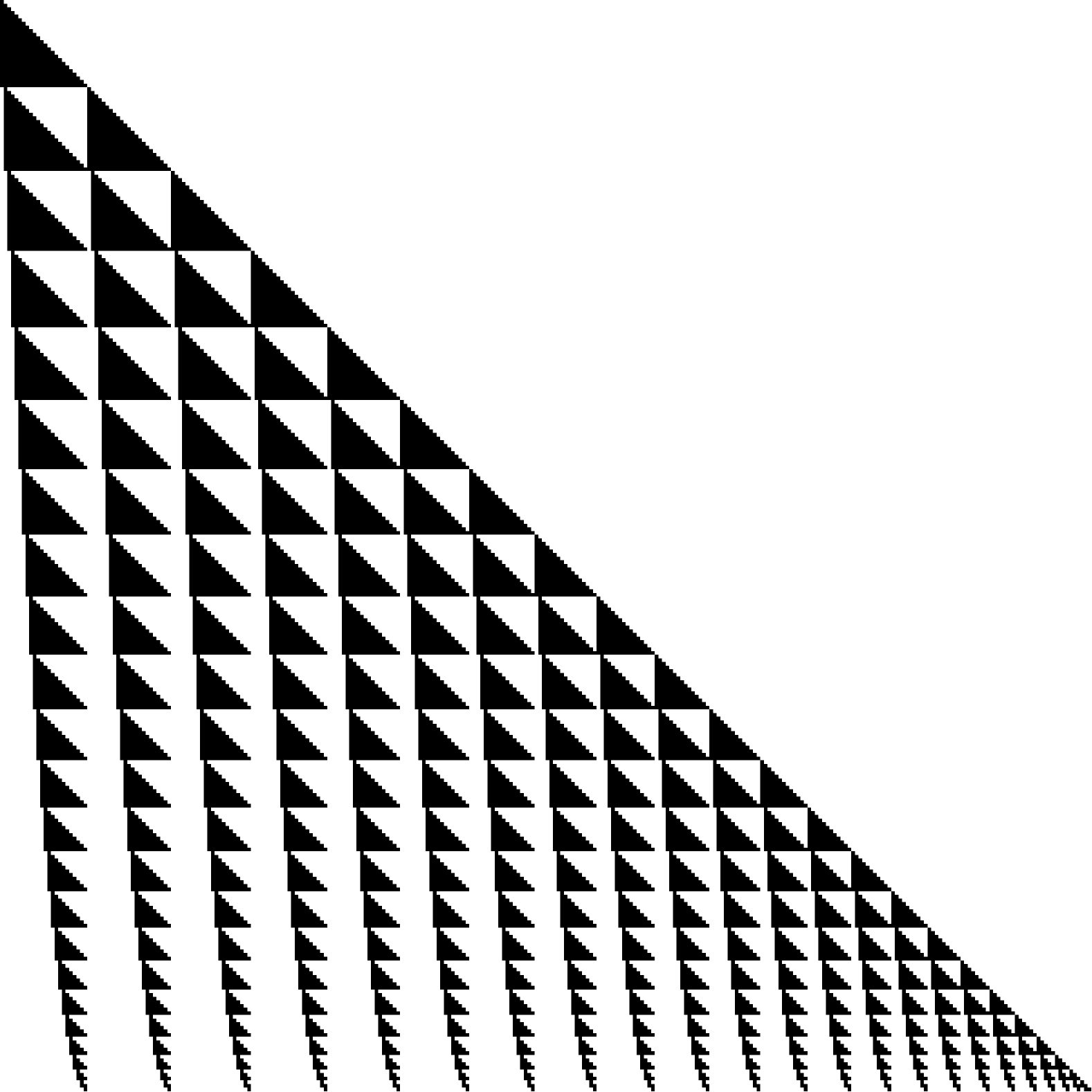}
    \caption*{The reduced formal context representing the Tamari lattice of order 22.}
\end{figure}

\newpage

\tableofcontents

\section*{Introduction}

\subsection*{Context}% Informal Context

The philosophical underpinnings of Formal Concept Analysis (FCA) has its roots in Arnauld and Nicole’s Port-Royal logic: Their seminal 1662 work ``La Logique ou l’art de penser'' explores the duality between extension and intension in the definition of concepts, insofar that a concept is defined both by the set of things it applies to (its extension) and by the set of properties that characterize it (its intension). It is this duality that is mirrored in the formalism of FCA.

Developed in the early 1980s, FCA formalizes this notion of a concept as a pair consisting of an extent (a set of objects) and an intent (a set of attributes) that are mutually maximal with respect to one another \cite{Wille1982Restructuring,Ganter, ganter1998mathematical, ganter2005fca}. The collection of all such concepts forms a complete lattice, known as the \emph{concept lattice}. The aforementioned duality is concretely realized as a Galois connection between sets of objects and sets of attributes, and draws  upon earlier theoretical work of Birkhoff \cite{birkhoff1973lattice}.

FCA has seen use in real-world data analysis and information sciences. We refer the reader to \cite{Poelmans2013FCA, Poelmans2013Applications} for a comprehensive literature review on where FCA has been applied; including applications in software mining, web analytics, medicine, biology, and chemistry. In this paper, however, we reverse the flow of knowledge and present a novel application of FCA to an emerging field of pure mathematics commonly referred to as \emph{homotopical combinatorics} which inhabits a rich intersection of equivariant homotopy theory and combinatorics (see \cite{blumberg2024homotopical} for an overview). 

Explicitly, fix a finite group $G$ and a $G$-space $X$ which comes equipped with a multiplicative structure. One then considers the multitude of ways in which the multiplication interacts with the $G$-action in a homotopical setting. Work of Blumberg--Hill informs us that these possibilities are controlled by (homotopy classes of) \emph{$N_\infty$ operads} for $G$ \cite{blumberghill}.  Further work of many authors proves that these homotopy classes are realized by a discrete structure called a \emph{transfer system} \cite{bbr, BP_operads, GW_operads, Rubin_comb}. Critical to the thrust of this paper is the fact that the collection of all transfer systems for $G$ naturally forms a finite lattice, which we henceforth denote $\mathsf{Tr}(\Sub(G))$.

In practice, $\mathsf{Tr}(\Sub(G))$ is only completely understood as a lattice for cyclic groups of prime power order, where the first author along with Barnes and Roitzheim prove that $\mathsf{Tr}(\Sub(C_{p^n}))$ is the $(n+1)$-st Tamari lattice \cite{bbr}. In the literature there are moreover enumeration results for abelian $p$-groups of rank 2, groups of the form $C_{p^nq}$ and $D_{p^n}$, and a handful of miscellaneous groups \cite{bao2023transfersystemsrankelementary, BMO_enumeration}.

\subsection*{Main results}

The main result of this paper is a novel contextualization of the computation of $\mathsf{Tr}(\Sub(G))$ in the realm of FCA. We appeal to the observation that any finite lattice $(L, \leq)$ is determined up to isomorphism by the restriction of the binary relation ${\leq} \subseteq L \times L$ to the subset $J(L) \times M(L)$ (see \cref{thm:FTFCA2}), where $J(L)$ and $M(L)$ are the subsets of meet and join-irreducible elements of $L$. The importance of this result is that one can, surprisingly, very easily determine $J(\mathsf{Tr}(\Sub(G)))$, $M(\mathsf{Tr}(\Sub(G)))$, and the restricted relation between them, using only the structure of the subgroup lattice of $G$:

\begin{letterthm}[\cref{sec:joinmeet}]
    Let $L$ be a finite lattice with $G$-action (e.g., $L = \Sub(G)$ for a finite group $G$). Then the reduced formal context of $\Tr(L)$ has attributes and objects both in bijective correspondence with the $G$-orbits of the set of nontrivial relations in $L$. The full data of the reduced formal context is described explicitly in \Cref{thm:relation}.
\end{letterthm}

With this structural result in place, we are then able to appeal to the wealth of machinery that FCA has to offer. The first of which are the computational tools which can be used to compute the number of concepts. We provide details of these computations in \Cref{app:A}, but for now we advertise the following calculation, which far out-scales any computation seen in the homotopical combinatorics literature to date, and can be obtained in approximately 2 hours of computational time on a personal laptop:

\begin{letterthm}[\cref{sec:joinmeet}]
    There are exactly $37,799,146,070$ transfer systems for $A_6$.
\end{letterthm}

Prior to this result, the largest calculation announced was by the first author and collaborators in \cite{basis_paper, balchin2025ninftysoftwarepackagehomotopical} where $\abs{\mathsf{Tr}(\Sub(S_5))} = 183,598,202$ was computed. We note here, and will return to in \Cref{app:A}, that this calculation took the use of high-performance computing facilities and a week of computation, whereas the methods derived in this paper allow the same computation to be achieved in under a minute.

Next, we move to the question of using FCA to form bounds for $\abs{\mathsf{Tr}(\Sub(G))}$ for several key families of groups. The key invariant we will use is the (co)density of a formal context. Briefly, any formal context can be represented as a binary matrix; the density $\delta$ is the proportion of 1's in the matrix, while the codensity $\rho$ is the proportion of 0's. There exist results providing non-trivial bounds for the number of concepts of a formal context given the density, whence the relevance of these calculations (see \cref{sec:complex} for details). Our first collection of results regards the codensity of the reduced formal context arising from the transfer lattice of cyclic groups. In what follows we identify $\Sub(C_N)$ with $[n_1] \times \cdots \times [n_k]$ where the $n_i$ are the exponents appearing in the prime factorization of $N$.

\begin{letterthm}[\cref{sec:tamari} and \cref{sec:cyclic}]
    For any natural numbers $n_1, \dots, n_k$,
    \[\rho(\Tr([n_1] \times \dots \times [n_k])) = \frac{\sum_{0 \leq x_1 \leq y_1 \leq z_1 \leq n_1} \dots \sum_{0 \leq x_k \leq y_k \leq z_k \leq n_k} \left(\prod_{i=1}^k (z_i - x_i +1) - \prod_{i=1}^k (z_i - y_i +1)\right)}{\left(\sum_{0 \leq x_1 \leq n_1} \dots \sum_{0 \leq x_k \leq n_k} \left(\prod_{i=1}^k (n_i - x_i + 1) - 1 \right)\right)^2}.\]
    In particular we are able to conclude that
    \begin{align*}\rho(\Tr([n])) &= \frac{(n+2)(n+3)}{6n(n+1)}\\[5pt]
    \rho(\Tr([n] \times [m])) &= \frac{(m+2) (m+3) (n+2) (n+3) (3 m n+4 m+4 n)}{36 (m+1) (n+1) (2m + 2n + m n)^2}\\[5pt]
    \rho(\Tr([1]^k)) &= \frac{6^k - 5^k}{(3^k-2^k)^2}.
    \end{align*}
\end{letterthm}

Next, we compute the codensity for all elementary abelian $p$-groups:

\begin{letterthm}[\cref{sec:elementary}]
    For any prime number $p$ and any positive integer $n$,
    \[
    \rho(\Tr(\Sub(C_p^n)))  = \dfrac{\sum_{i=0}^{n-1} \sum_{j=1}^{n-i} \sum_{k=0}^{n-i-j} \qbinom{n}{i}{p} \qbinom{n-i}{j}{p} \qbinom{n-i-j}{k}{p} (a_{j+k} - a_k)}{\left( \sum_{i=0}^{n-1} \qbinom{n}{i}{p} (a_{n-i}-1) \right)^2}
    \]
    where $\binom{n}{d}_p$ denotes a Gaussian binomial coefficient (giving the number of $d$-dimensional subspaces of $\mathbb{F}_p^n$) and
    \vspace{-2mm}
    \[a_d = \sum_{i=0}^d \qbinom{d}{i}{p}.\]
    \vspace{-2mm}
    Consequently,
    \vspace{-3mm}
    \[\lim_{p \to \infty}\rho(\Tr(\Sub(C_p^n))) = \begin{cases} 1 &: n = 1 \\ 1/4 &: n = 2 \\ 0 &: n > 2 \end{cases} \qquad \qquad \text{and}\qquad \qquad \lim_{n \to \infty}\rho(\Tr(\Sub(C_p^n))) = 0.\]
\end{letterthm}

\subsection*{Outline of the document}

This paper has been written with the homotopical combinatorial audience in mind. Hence, in \Cref{sec:primer} we spend some time introducing the required results from FCA. In \Cref{sec:joinmeet} we discuss a characterization of the join and meet-irreducible elements of $\Tr(\Sub(G))$ as well as the restricted relation between them required to obtain the promised formal context. Next, in \Cref{sec:complex} we relate the notion of \emph{complexity} as introduced in \cite{basis_paper} to the notion of contranomial scales in FCA, which thus proves that complexity is a computationally complex invariant. \Cref{sec:complex} also introduces upper bounds for the size of the concept lattice in terms of the density.

\cref{sec:tamari} explores the formal context for $C_p^n$, which by the results of \cite{bbr} is equivalently the formal contexts for the family of Tamari lattices. This provides us with a warm-up for the results of \cref{sec:cyclic} which computes the codensity of the formal context for all cyclic groups. In \cref{sec:elementary} we generalize in the orthogonal direction and instead compute the codensity of the formal context for elementary abelian $p$-groups.

Finally, in \cref{sec:sat} we explore the theory of FCA for the \emph{saturated} and \emph{cosaturated} transfer systems which have proved to be vital in many applications of homotopical combinatorics to equivariant algebra.

\subsection*{Notation}

\begin{itemize}
    \item For a finite lattice $(L, \leq)$, we write
    \begin{itemize}
        \item $J(L)$ for set of join-irreducibles in $L$.
        \item $M(L)$ for set of meet-irreducibles in $L$.
        \item $\delta(L)$ for density of the $(J(L),M(L),\leq)$ formal context for $L$.
        \item $\rho(L)$ for codensity of the $(J(L),M(L),\leq)$ formal context for $L$.
        \item $\operatorname{Rel}^*(L)$ for the collection of non-trivial relations in $L$.
    \end{itemize}
    \item $\qbinom{n}{k}{q}$ for $q$-binomial coefficients (sometimes referred to as \emph{Gaussian coefficients}).
    \item $[n]$ for the totally ordered lattice $\{0 < \dots < n\}$.
    \item $\Sub(G)$ for the subgroup lattice of a finite group $G$ equipped with the conjugation action.
    \item $\Tr(L)$ for lattice of transfer systems on a lattice $L$ with $G$-action (e.g., $\Sub(G)$).
    \item $A \to B$ as a synonym for either an ordered pair  $(A,B)$ which may or may not appear in a relation or a morphism from $A$ to $B$ in a preorder, as clear from context.
    \item $\floor{A \to B}$ for the transfer system generated by $A \to B$.
    \item $\ceil{A \to B}$ for the cotransfer system generated by $A \to B$.
    \item $X/\!/G$ for the set of orbits of a $G$-set $X$.
\end{itemize}

\subsection*{Acknowledgments}

The authors thank Bernhard Ganter for helping us navigate the FCA literature, and Uta Priss for maintaining \url{https://upriss.github.io/fca/fca.html}. The authors are also grateful to Sridhar Ramesh for correcting an error in the statement of \Cref{thm:FTFCA2}.

The authors would also like to thank the Isaac Newton Institute for Mathematical Sciences, Cambridge, for support and hospitality during the programme ``Equivariant homotopy theory in context" where work on this paper was undertaken. This work was supported by EPSRC grant no EP/Z000580/1. The second author was partially supported by NSF grant DMS-2414922.

\section{Primer on Formal Concept Analysis}\label{sec:primer}

Formal concept analysis (FCA) is a method of data analysis which is designed to describe relationships between a set of objects and a set of attributes which arise in real-world situations. We begin by introducing the relevant definitions.

\begin{definition}% formal context
    A \emph{formal context} is an ordered triple $(X, Y, R)$ where $X$ and $Y$ are non-empty sets and $R \subseteq X \times Y$ is a relation. The elements $x \in X$ are called \emph{objects} and the elements $y \in Y$ are called \emph{attributes}.
\end{definition}

\begin{definition}
    Let $(X, Y, R)$ be a formal context. Define two operators $(-)^\uparrow \colon \mathcal{P}(X) \to \mathcal{P}(Y)$ and $(-)^\downarrow \colon \mathcal{P}(Y) \to \mathcal{P}(X)$ as
    \begin{align*}
        A^\uparrow &= \{y \in Y : \text{for each } x \in A,\, (x,y) \in R\},\\
        B^\downarrow &= \{x \in Y : \text{for each } y \in B,\, (x,y) \in R\}.
    \end{align*}
\end{definition}

\begin{definition}
    Let $(X, Y, R)$ be a formal context. A \emph{formal concept} in $(X,Y,R)$ is a pair $A \subseteq X$, $B \subseteq Y$ such that $A^\uparrow = B$ and $B^\downarrow = A$.

    Denote by $\conlat(X, Y, R)$ the collection of all formal concepts of $(X, Y, R)$. It is naturally partially ordered by $(A_1,B_1) \leq (A_2,B_2) \iff A_1 \subseteq A_2$ (equivalently if and only if $B_1 \supseteq B_2$).
\end{definition}

A key result of FCA is that the poset $\conlat(X, Y, R)$ is a complete lattice, and hence goes by the name of the \emph{concept lattice}:

\begin{theorem}[Fundamental theorem of FCA, Part I \cite{Wille1982Restructuring}]\label{thm:FTFCA1}
    Let $(X, Y, R)$ be a formal context. Then $\conlat(X, Y, R)$ is a complete lattice with meet and join given by
    \begin{align*}
        \bigwedge_{j \in J} (A_j,B_j) &= (\bigcap_{j \in J} A_j, (\bigcup_{j \in J}B_j)^{\downarrow \uparrow}) \\
        \bigvee_{j \in J} (A_j,B_j) &= ((\bigcup_{j \in J}A_j)^{\uparrow \downarrow}, \bigcap_{j \in J} B_j)
    \end{align*}
\end{theorem}

\begin{example}\label{ex:exampleFCA}
In \cref{fig:example context} we pick a particular formal context and produce its concept lattice. Note that any formal context can be represented by a binary matrix by picking total orderings on $X$ and $Y$. Equivalently, we may display the context as an image of black and white pixels where a white pixel represents a 1, and a black pixel represents a 0:

\begin{figure}[h!]
    \centering
    \[
    \begin{array}{c|c|c|c|c}
           & y_1 & y_2 & y_3 & y_4 \\ \hline
       x_1 & 1   & 0   & 0   & 0   \\ \hline
       x_2 & 0   & 0   & 0   & 1   \\ \hline
       x_3 & 1   & 1   & 1   & 0
    \end{array}
    \qquad\qquad
    \begin{tikzpicture}[baseline=4em]
        \fill [black] (0,3) circle (0.1);
        \fill [black] (-1,2) circle (0.1);
        \fill [black] (-1,1) circle (0.1);
        \fill [black] (1,1.5) circle (0.1);
        \fill [black] (0,0) circle (0.1);
        \draw (0,0) -- (-1,1) -- (-1,2) -- (0,3) -- (1,1.5) -- (0,0);
    \end{tikzpicture}
    \qquad\qquad
    \begin{tikzpicture}[anchor=base, baseline=-1em]
        \draw[nice-pink, line width=2pt] (-0.0352,-1.0352) rectangle ++ (2.07,1.57);
        \fill[black] (0.5, 0.0) rectangle ++(0.51, 0.51);
        \fill[black] (1.0, 0.0) rectangle ++(0.51, 0.51);
        \fill[black] (1.5, 0.0) rectangle ++(0.51, 0.51);
        \fill[black] (0.0, -0.5) rectangle ++(0.51, 0.51);
        \fill[black] (0.5, -0.5) rectangle ++(0.51, 0.51);
        \fill[black] (1.0, -0.5) rectangle ++(0.51, 0.51);
        \fill[black] (1.5, -1.0) rectangle ++(0.51, 0.51);
    \end{tikzpicture}
    \]
    \caption{An example of a finite formal context, the corresponding finite lattice, and the pictorial representation of the context.}
    \label{fig:example context}
\end{figure}
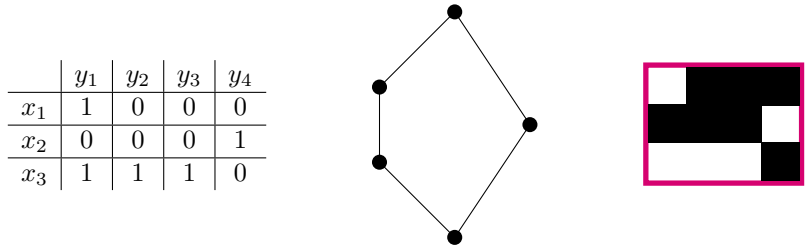

In \Cref{fig:example context}, the minimum element of the lattice is the formal concept $(\varnothing, \{y_1, y_2, y_3, y_4\})$. The maximum element of the lattice is the formal concept $(\{x_1, x_2, x_3\}, \varnothing)$. The other three formal concepts are $(\{x_3\}, \{y_1, y_2, y_3\})$, $(\{x_1,x_3\}, \{y_1\})$, and $(\{x_2\}, \{y_4\})$.
\end{example}

Note that the context of \cref{ex:exampleFCA} contains redundancies -- that is, the second and third columns are equal.  In other words, we have two objects with exactly the same attributes. This duplication of objects has no effect on the resulting lattice; one can remove a duplicated column with no effect on the resulting lattice. There is a more general fact: one can remove any row or column which is the intersection of other rows or columns.

\begin{definition}
    A context $(X,Y,R)$ is \emph{reduced} if no row (resp., column) can be obtained by intersecting other rows (resp., columns).
\end{definition}

Given any finite context, one can \emph{reduce} it by iteratively removing offending rows and columns until a reduced context is obtained.

\begin{proposition}[{\cite[Proposition 14]{FCAbook}}]
    Let $(X,Y,R)$ be a finite context and let $(X',Y',R)$ be its reduction. Then there is a lattice isomorphism
    \[
    \conlat(X, Y, R) \cong \conlat(X', Y', R).
    \]
    That is, the formation of the concept lattice is insensitive to taking the reduction.
\end{proposition}

Important to this current paper is the somewhat surprising result is that there exists a converse to \cref{thm:FTFCA1}. That is, every finite lattice arises as the concept lattice of some finite formal context. As suggested by the above proposition, this context is not uniquely determined, but there is a unique \emph{reduced} formal context representing each finite lattice.

Recall that for $L$ a lattice, an element $x \in L$ is said to be \emph{join-irreducible} (resp., \emph{meet-irreducible}) if, for all $X \subseteq L$, $x = \bigvee X$ implies $x \in X$ (resp., $x = \bigwedge X$ implies $x \in X$). We denote by $J(L)$ the set of join-irreducible elements of $L$, and by $M(L)$ the collection of meet-irreducible elements of $L$. By definition, $J(L) = M(L^\mathrm{op})$ and $M(L) = J(L^\mathrm{op})$.

\begin{theorem}[The Fundamental Theorem of FCA, Part II \cite{Wille1982Restructuring}]\label{thm:FTFCA2}
    Let $L$ be a complete lattice such that each element of $L$ can be expressed both as a join of join-irreducible elements and as a meet of meet-irreducible elements. Then there is a canonical isomorphism $L \to \conlat(J(L), M(L), {\leq})$ given by
    \[z \mapsto (\{x \in J(L) : x \leq z\},\{y \in M(L) : z \leq y\}).\]

    Moreover, the context $(J(L), M(L), {\leq})$ is reduced, and is the unique (up to isomorphism) reduced context with concept lattice isomorphic to $L$.
\end{theorem}

In what follows, we will be interested only in finite lattices. All finite lattices satisfy the hypothesis of \cref{thm:FTFCA2}, and in a finite lattice the join and meet-irreducible elements can be characterized finitarily.

\begin{proposition}
    Let $L$ be a finite lattice with minimum element $\bot$ and maximum element $\top$. An element $x \in L$ is:
    \begin{enumerate}
        \item join-irreducible if and only if $x \neq \bot$ and, for all $a, b \in L$, $x = a \vee b$ implies $x = a$ or $x = b$;
        \item meet-irreducible if and only if $x \neq \top$ and, for all $a, b \in L$, $x = a \wedge b$ implies $x = a$ or $x = b$.
    \end{enumerate}
\end{proposition}

\section{Join and Meet-Irreducible Transfer Systems}\label{sec:joinmeet}

In \cref{sec:primer} we introduced the basics of FCA needed for this paper. In this section we will briefly recall the necessary machinery for transfer and cotransfer systems. We will then give explicit descriptions of the join and meet-irreducible elements of the lattice of transfer systems as well as the inclusion relation between them which will then allow us to appeal to \cref{thm:FTFCA2}.

We will use the term ``$G$-lattice'' to refer to a lattice $L$ equipped with an action of a \emph{finite} group $G$ by lattice automorphisms. The key example of interest is $\Sub(G)$, the subgroup lattice of $G$, on which $G$ acts by conjugation.

\begin{definition}
    Let $(L,\leq)$ be a $G$-lattice. A \emph{$G$-relation} is a partial order $\to$ on $L$ which refines $\leq$ (i.e., $x \to y$ implies $x \leq y$) and respects the action of $G$ (i.e., $x \to y$ implies $g \cdot x \to g \cdot y$ for all $g \in G$). A $G$-relation is a:
    \begin{itemize}
        \item \emph{transfer system} if whenever $x \to y$ and $x' \leq y$, we have $x \wedge x' \to x'$;
        \item \emph{cotransfer system} if whenever $x \to y$ and $x \leq y'$, we have $y' \to y \vee y'$.
    \end{itemize}
\end{definition}

The naming here reflects the fact that cotransfer systems on $L$ are in canonical bijective correspondence with transfer systems on $L^\mathrm{op}$.

\begin{warning}\label{warn:trL=coTrLop}
    Henceforth, we identify cotransfer systems on $L$ with transfer systems on $L^\mathrm{op}$. Strictly speaking, a transfer system on $L^\mathrm{op}$ is a subcategory of $L^\mathrm{op}$, whereas a cotransfer system on $L$ is a subcategory of $L$. The canonical identification simply reverses all of the arrows in a (co)transfer system.
\end{warning}

\begin{lemma}[{\cite[Proposition 3.7]{fooqw}}]\label{lem:transferlattice}
    Let $L$ be a $G$-lattice. The collection $\Tr(L)$ of transfer systems (resp., $\mathsf{coTr}(L)$ of cotransfer systems) on $L$ is partially ordered by refinement (i.e., $\subseteq$). With this ordering, $\Tr(L)$ (resp., $\coTr(L)$) is a complete lattice whose meet operation is given by intersection.
\end{lemma}

A key problem in homotopical combinatorics is to understand the lattice $\Tr(\Sub(G))$ for finite groups of interest. What we shall see is that it is somewhat easy to describe the join and meet-irreducible elements of $\Tr(\Sub(G))$ without having to compute any transfer systems, and we moreover have a characterization of when a particular join-irreducible will be contained in a particular meet-irreducible. This will then allow us to appeal to \Cref{thm:FTFCA2} to describe the concept lattice associated to $\Tr(\Sub(G))$ in terms of $\Sub(G)$.

\begin{definition}\label{def:cogen}
    Let $(L,\leq)$ be a $G$-lattice, and let $x, y \in L$ be elements such that $x \leq y$. We denote by:
    \begin{itemize}
        \item $\floor{x \to y}$ the smallest transfer system $\to$ such that $x \to y$, i.e., the transfer system generated by the pair $(x,y)$.
        \item $\ceil{x \to y}$ the smallest cotransfer system $\to$ such that $x \to y$, i.e., the cotransfer system generated by the pair $(x,y)$.
    \end{itemize}
\end{definition}

We will now show that the join-irreducible elements of $\Tr(L)$ are precisely those of the form $\floor{x \to y}$ with $x < y$ (and likewise for $\coTr(L)$), which will be \cref{prop:join_irr,prop:meet_irr}. To best classify these join-irreducibles, we record the following lemma, where we use \cite[Theorem A.2]{rubin} (which is stated specifically for the case $L = \Sub(G)$, but holds in this greater generality). Since $\coTr(L) = \Tr(L^\mathrm{op})$, we will focus on $J(\Tr(L))$ and deduce the structure of $J(\coTr(L))$ as a formal consequence. In the case that $L$ has trivial $G$-action, this classification is due to Luo and Rognerud \cite[Proposition 1.3]{luo2024latticeweakfactorizationsystems}; our work here serves only to generalize their result to the case of an arbitrary $G$-lattice.

\begin{lemma}\label{lem:genconj}
    Let $(L, \leq)$ be a $G$-lattice, and let $x,x',y,y' \in L$ be elements such that $x < y$ and $x' < y'$. The following are equivalent:
    \begin{enumerate}
        \item $\floor{x \to y} = \floor{x' \to y'}$;
        \item There is some $g \in G$ such that $g \cdot x = x'$ and $g \cdot y = y'$.
    \end{enumerate}
\end{lemma}
\begin{proof}
    We use Rubin's description of $\floor{x \to y}$ from \cite[Theorem A.2]{rubin}, which is as follows: let
    \begin{align*}
        \mathcal{X}_0 &= \{x \to y\} & \mathcal{X}'_0 &= \{x' \to y'\} \\
        \mathcal{X}_1 &= \{g \cdot x \to g \cdot y : g \in G\} & \mathcal{X}'_1 &= \{g \cdot x' \to g \cdot y' : g \in G\} \\
        \mathcal{X}_2 &= \{(g \cdot x) \wedge z \to z : g \in G, z \leq g \cdot y\} & \mathcal{X}'_2 &= \{(g \cdot x') \wedge z \to z : g \in G, z \leq g \cdot y'\} \\
        \mathcal{X}_3 &= \mathcal{X}_2^\circ & \mathcal{X}'_3 &= {\mathcal{X}'_2}^\circ
    \end{align*}
    where ${(-)}^\circ$ denotes the reflexive-transitive closure of a binary relation. Then $\floor{x \to y} = \mathcal{X}_3$ and $\floor{x' \to y'} = \mathcal{X}'_3$.

    From this description, it is clear that (2) implies (1): if (2) holds, then already $\mathcal{X}_1 = \mathcal{X}'_1$, so also $\mathcal{X}_3 = \mathcal{X}'_3$. It remains to show that (1) implies (2), so assume $\floor{x \to y} = \floor{x' \to y'}$.

    Now $x' \to y' \in \floor{x \to y}$ implies (from the definitions of $\mathcal{X}_i$'s above) that $y'$ is subconjugate to $y$, i.e., there exists some $g \in G$ such that $y' \leq g \cdot y$. Symmetrically, there exists some $g' \in G$ such that $y \leq g' \cdot y'$. Together, this yields $y' \leq (gg') \cdot y'$. Since $G$ is finite, we have
    \[y' \leq (gg') \cdot y' \leq (gg')^2 \cdot y' \leq \dots \leq (gg')^{\abs{G}} \cdot y' = y'\]
    and thus $y' = (gg') \cdot y'$, whence
    \[y' \leq g \cdot y \leq (gg') \cdot y' = y',\]
    so $y' = g \cdot y$.

    We now know that $x' \to g \cdot y \in \floor{x \to y}$, so there exists a sequence
    \[((g_1 \cdot x) \wedge z_1 \to z_1, \dots, (g_n \cdot x) \wedge z_n \to z_n)\]
    of elements of $\mathcal{X}_2$ (where $n$ is a positive integer) whose composite is $x' \to g \cdot y$. In other words, we have a sequence $g_1, \dots, g_n$ of elements of $G$ and a sequence $z_1, \dots, z_n$ of elements of $L$ such that:
    \begin{enumerate}[label=(\roman*)]
        \item $z_i \leq g_i \cdot y$ for all $1 \leq i \leq n$,
        \item $(g_1 \cdot x) \wedge z_1 = x'$,
        \item $z_i = (g_{i+1} \cdot x) \wedge z_{i+1}$ for all $1 \leq i < n$,
        \item $z_n = g \cdot y$.
    \end{enumerate}
    We will now show that $x' = g_i \cdot x \leq z_i$ for all $1 \leq i \leq n$ by induction.
    \begin{description}
        \item[Base Case] In particular, (ii) implies that $x' \leq g_1 \cdot x$. Symmetrically, we obtain $x \leq g'_1 \cdot x'$ for some $g'_1 \in G$. By the same argument as before, this implies $x' = g_1 \cdot x$. Now (ii) says that $(g_1 \cdot x) \wedge z_1 = (g_1 \cdot x)$, so $g_1 \cdot x \leq z_1$.
        \item[Inductive Step] Suppose $x' = g_i \cdot x \leq z_i$ for some $1 \leq i < n$. Then (iii) says that
        \[g_i \cdot x \leq z_i = (g_{i+1} \cdot x) \wedge z_{i+1},\]
        so in particular $g_i \cdot x \leq g_{i+1} \cdot x$. By the same argument as before, this implies $g_i \cdot x = g_{i+1} \cdot x$. Now
        \[g_{i+1} \cdot x = g_i \cdot x \leq z_i = (g_{i+1} \cdot x) \wedge z_{i+1} \leq z_{i+1}\]
        yields $x' = g_i \cdot x = g_{i+1} \cdot x \leq z_{i+1}$, as desired.
    \end{description}
    Thus we have $x' = g_n \cdot x$. By (i) and (iv) we have $g \cdot y \leq g_n \cdot y$, which implies $g \cdot y = g_n \cdot y$. Thus, $x' = g_n \cdot x$ and $y' = g \cdot y = g_n \cdot y$, as desired.
\end{proof}

\begin{lemma}\label{lem:join irr implies principal}
    Let $L$ be a $G$-lattice and let $\mathsf{T} \in J(\Tr(L))$. Then $\mathsf{T} = \floor{x \to y}$ for some $x, y \in L$ with $x < y$.
\end{lemma}
\begin{proof}
    Let $S = \{\floor{x \to y} : x \xrightarrow{\mathsf{T}} y\}$. Then $\mathsf{T} = \bigvee S$, so $\mathsf{T} = \floor{x \to y}$ for some $x, y \in L$ with $x \leq y$. If $x = y$ then $\mathsf{T}$ is the minimum element of $\Tr(L)$, contradicting the assumption that $\mathsf{T}$ is join-irreducible.
\end{proof}

\begin{lemma}\label{lem:principal implies join irr}
    Let $L$ be a $G$-lattice and let $x,y \in L$ with $x < y$. Then
    $\floor{x \to y} \in J(\Tr(L))$.
\end{lemma}
\begin{proof}
    Suppose
    \[\floor{x \to y} = \bigvee \mathcal{T}\]
    for some $\mathcal{T} \subseteq \Tr(L)$. By \cite[Theorem A.2]{rubin}, since $\bigcup \mathcal{T}$ is closed under the action on $G$, we have
    \[\bigvee \mathcal{T} = \left(\bigcup_{\mathsf{T} \in \mathcal{T}} \bigcup_{a \to b \in \mathsf{T}} \{a \wedge z \to z : z \leq b\}\right)^\circ,\]
    where again $({-})^\circ$ denotes the reflexive-transitive closure of a binary relation. Since $x < y$, we have a sequence $\mathsf{T}_1, \dots, \mathsf{T}_n$ of elements of $\mathcal{T}$ and sequences $a_1, \dots, a_n$, $b_1, \dots, b_n$, and $z_1, \dots, z_n$ of elements of $L$ (where $n$ is a positive integer) such that:
    \begin{enumerate}[label=(\roman*)]
        \item $a_i \xrightarrow{\mathsf{T}_i} b_i$ for all $1 \leq i \leq n$,
        \item $z_i \leq b_i$ for all $1 \leq i \leq n$,
        \item $a_1 \wedge z_1 = x$,
        \item $z_i = a_{i+1} \wedge z_{i+1}$ for all $1 \leq i < n$,
        \item $z_n = y$.
    \end{enumerate}
    We may also assume (without loss of generality) that $a_i \neq b_i$ for all $1 \leq i \leq n$, since if $a_i = b_i \geq z_i$ then $a_i \wedge z_i \to z_i$ is an identity relation. With this in mind, since $\mathsf{T} \subseteq \floor{x \to y} = \{(g \cdot x) \wedge w \to w : g \in G, w \leq g \cdot y\}^\circ$ for all $\mathsf{T} \in \mathcal{T}$, (i) implies that there are sequences of elements $g_1, \dots, g_n$ and $g'_1, \dots g'_n$ of $G$ such that
    \begin{enumerate}[label=(\roman*)]
        \setcounter{enumi}{5}
        \item $a_i \leq g_i \cdot x$ for all $1 \leq i \leq n$,
        \item $b_i \leq g'_i \cdot y$ for all $1 \leq i \leq n$.
    \end{enumerate}
    We will now show that $z_i \geq x = a_i$ for all $1 \leq i \leq n$ by induction.
    \begin{description}
        \item[Base Case] From (iii) and (vi), we have
        \[x = a_1 \wedge z_1 \leq a_1 \leq g_1 \cdot x,\]
        which implies (as in previous arguments) that $x = g_1 \cdot x$, and (directly) that $z_1 \geq x$. Then the same compound inequality shows that $a_1 = x$, as desired.
        \item[Inductive Step] Suppose $z_i \geq x = a_i$ for some $1 \leq i < n$. Then from (iv) and (vi) we have
        \[x \leq z_i = a_{i+1} \wedge z_{i+1} \leq a_{i+1} \leq g_{i+1} \cdot x,\]
        so $x = g_{i+1} \cdot x$ and $z_{i+1} \geq x$. The same compound inequality now shows $a_{i+1} = x$.
    \end{description}
    Now we know in particular that $a_n = x$. From (ii), (v), and (vii), we obtain
    \[y = z_n \leq b_n \leq g'_n \cdot y,\]
    so $y = g'_n \cdot y = b_n$. Thus, $x \to y = a_n \to b_n \in \mathsf{T}_n$. We thus obtain
    \[\floor{x \to y} \subseteq \mathsf{T}_n \subseteq \floor{x \to y},\]
    and therefore $\mathsf{T}_n = \floor{x \to y}$, as desired.
\end{proof}

We are now in a position to describe the join-irreducible elements of $\mathsf{Tr}(L)$.

\begin{proposition}\label{prop:join_irr}
    Let $L$ be a $G$-lattice. Then the set of join-irreducible elements of $\Tr(L)$ is equal to the set of transfer systems generated by a single non-identity relation, i.e.,
    \[
    J(\Tr(L)) = \{ \floor{x \to y} : x < y\}.
    \]
    Moreover, there is a bijection between $J(\Tr(L))$ and the set $\frac{\{(x,y) : x<y\}}{(x,y) \sim (g \cdot x, g \cdot y)}$. In particular, the number of join-irreducible transfer systems on $L$ is equal to $\abs{\operatorname{Rel}^*(L)/\!/G}$.
\end{proposition}

\begin{proof}
    Apply \cref{lem:join irr implies principal}, \cref{lem:principal implies join irr}, and \cref{lem:genconj}.
\end{proof}

Since cotransfer systems on $L$ are the same as transfer systems on $L^\mathrm{op}$, we immediately obtain a similar classification of join-irreducible cotransfer systems.

\begin{corollary}
    Let $L$ be a $G$-lattice. Then
    \[
    J(\coTr(L)) = \{\ceil{x \to y} : x < y\} \cong \frac{\{(x,y) : x<y\}}{(x,y) \sim (g \cdot x, g \cdot y)}.
    \]
\end{corollary}

We now move to describing the meet-irreducible elements. This requires a little more work, but the general strategy is already outlined in \cite{luo2024latticeweakfactorizationsystems}. By general abstract theory, the meet-irreducible elements of $\Tr(L)$ are in bijection with the join-irreducible elements of $\Tr(L)^\mathrm{op}$. We can exploit this to give a general description of the meet-irreducible elements. We begin with a preliminary definition of lifting conditions.

\begin{definition}
    Let $\mathcal{C}$ be a category, and let $i \colon A \to B$ and $p \colon X \to Y$ be morphisms in $\mathcal{C}$. If for every $f$ and $g$ that make the square
\begin{equation}\label{eq:lift}\begin{tikzcd}
	A & X \\
	B & Y
	\arrow["g", from=1-1, to=1-2]
	\arrow["i"', from=1-1, to=2-1]
	\arrow["p", from=1-2, to=2-2]
	\arrow["{\exists\lambda}"{description}, dotted, from=2-1, to=1-2]
	\arrow["f"', from=2-1, to=2-2]
\end{tikzcd}\end{equation}
commute, there exists a lifting $\lambda$ such that the two triangles commute, we say $i$ has the \emph{left lifting property} with respect to $p$, or equivalently, $p$ has the \emph{right lifting property} with respect to $i$.

For $S$ a class of morphisms in $\mathcal{C}$ we denote by 
\[
 S^{\boxslash} \coloneq \{p : p \text{ has the right lifting property with respect to all } i \in S\}.
\]
\end{definition}

\begin{lemma}\label{lem:cotransfer_duality}
    Let $L$ be a $G$-lattice. Then there is a poset isomorphism $\coTr(L) \xrightarrow{\sim} \Tr(L)^\mathrm{op}$ which sends a cotransfer system $\mathsf{C}$ to $\mathsf{C}^\boxslash$.
\end{lemma}
\begin{proof}
    In \cite[Lemma 3.1]{luo2024latticeweakfactorizationsystems}, it is shown that this holds when $L$ has trivial $G$-action. Let $L_\mathrm{triv}$ denote the $G$-lattice which is the same as $L$ but with trivial $G$-action. Then the $G$-action on $L$ induces a $G$-action on both $\Tr(L_\mathrm{triv})$ and $\coTr(L_\mathrm{triv})$, and in \cite[Lemma 9.4]{luo2024latticeweakfactorizationsystems} it is shown\footnote{In \cite{luo2024latticeweakfactorizationsystems} this is only stated for the special case $L = \Sub(G)$, but the result holds in this greater generality with the same proof.} that $\Tr(L) = \Tr(L_\mathrm{triv})^G$. The same holds for lattices of cotransfer systems, since
    \[\coTr(L) = \Tr(L^\mathrm{op}) = \Tr((L^\mathrm{op})_\mathrm{triv})^G = \Tr((L_\mathrm{triv})^\mathrm{op})^G = \coTr(L_\mathrm{triv})^G.\]
    The isomorphism
    \[\coTr(L_\mathrm{triv}) \xrightarrow{{(-)}^\boxslash} \Tr(L_\mathrm{triv})^\mathrm{op}\]
    is $G$-equivariant, so we obtain the desired isomorphism as the composite
    \[\coTr(L) = \coTr(L_\mathrm{triv})^G \xrightarrow{{(-)}^\boxslash} (\Tr(L_\mathrm{triv})^\mathrm{op})^G = (\Tr(L_\mathrm{triv})^G)^\mathrm{op} = \Tr(L)^\mathrm{op}.\qedhere\]
\end{proof}

\begin{proposition}\label{prop:meet_irr}
    Let $L$ be a $G$-lattice. Then the set of meet-irreducible elements of $\Tr(L)$ is equal to the set of right lifts of cotransfer systems generated by a single non-identity relation, i.e.,
    \[
    M(\Tr(L)) = \{ \ceil{x \to y}^{\boxslash} : x < y\}
    \]
    Moreover, there is a canonical bijection between $M(\Tr(L))$ and $J(\Tr(L))$. In particular, the number of meet-irreducible transfer systems on $L$ is equal to $\abs{\operatorname{Rel}^*(L)/\!/G}$.
\end{proposition}

\begin{proof}
    We have established
    \[M(\Tr(L)) = J(\Tr(L)^\text{op}) \cong J(\coTr(L)) = \{\ceil{x \to y} : x < y\},\]
    where $J(\coTr(L)) \xrightarrow{\sim} J(\Tr(L)^\text{op})$ sends $\ceil{x \to y}$ to $\ceil{x \to y}^\boxslash$. The remaining claim comes from the fact that $\floor{x \to y} \mapsto \ceil{x \to y}^\boxslash : J(\Tr(L)) \to M(\Tr(L))$ is a bijection (by \cref{lem:genconj}).
\end{proof}

Now that we have a good description of the join and meet-irreducible elements of $\Tr(L)$, we need to have an explicit description of the inclusions between them so that we may build the associated reduced formal context. This is the content of the following results. We first record an explicit description of the cotransfer system generated by a single non-trivial relation, which is formally dual to the description of $\floor{x \to y}$ above.

\begin{lemma}\label{cons:rubin}
    Let $L$ be a $G$-lattice, and let $x,y \in L$ such that $x \leq y$. Then
    \[\ceil{x \to y} = \{z \to (g \cdot y) \vee z : g \in G, z \geq g \cdot x\}^\circ\]
    where again $({-})^\circ$ denotes the reflexive-transitive closure of a binary relation.
\end{lemma}

\begin{theorem}\label{thm:relation}
    Let $L$ be a $G$-lattice, and let $a, b, x, y \in L$ be elements such that $a < b$ and $x < y$. Then
    \[
    \floor{a \to b} \subseteq \ceil{x \to y}^{\boxslash} \text{ if and only if, for all } g \in G,\, 
    g \cdot a \not\geq x \text{ or } g \cdot b \not\geq y \text{ or } g \cdot a \geq y.
    \]
\end{theorem}

\begin{proof}
    By consulting \cref{cons:rubin} and \eqref{eq:lift} we can describe all elements which have the right lifting property with respect to the cotransfer system $\ceil{x \to y}$. First, we note directly that an arrow $a \to b$ having the right lifting property with respect to an arrow $z \to (g \cdot y) \vee z$ means that, if $z \leq a$ and $(g \cdot y) \vee z \leq b$, then $(g \cdot y) \vee z \leq a$. Equivalently: if $z \leq a$ and $g \cdot y \leq b$ then $g \cdot y \leq a$. Thus, we have:
    \begin{align*}
    \ceil{x \to y}^{\boxslash}
    &= \bigcap_{g \in G} \{a \to b : \text{for all } z \geq g \cdot x, \text{ if } z \leq a \text{ and } g \cdot y \leq b \text{ then } g \cdot y \leq a\}\\
    &=\bigcap_{g \in G} \{a \to b : \text{for all } z \in L, g \cdot x \leq z \leq a \text{ and } g \cdot y \leq b \text{ implies } g \cdot y \leq a\}\\
    &=\bigcap_{g \in G} \{a \to b : g \cdot x \leq a \text{ and } g \cdot y \leq b \text{ implies } g \cdot y \leq a\}\\
    &= \bigcap_{g \in G} \{a \to b : g \cdot x \nleq a \text{ or } g \cdot y \nleq b \text{ or } g \cdot y \leq a\}\\
    &= \{a \to b : \text{for all } g \in G, x \nleq g \cdot a \text{ or } y \nleq g \cdot b \text{ or } y \leq g \cdot a\}.
    \end{align*}
    We then use that observation that $\floor{a \to b} \subseteq \ceil{x \to y}^\boxslash$ if and only if $a \to b \in \ceil{x \to y}^\boxslash$.
\end{proof}

\begin{summary}
    We now specialize to the case that $L = \Sub(G)$ equipped with the conjugation action. By \Cref{prop:join_irr,prop:meet_irr,thm:relation}, we obtain a reduced formal context $(X,Y,R)$ where
    \begin{align*}
    X = Y = \{(H,K) : H < K \leq G\}/\{(H,K) \sim (gHg^{-1}, gKg^{-1})\}\\
    (A,B) R (H,K) \iff \forall g \in G, gAg^{-1} \ngeq H \text{ or } gBg^{-1} \ngeq K \text{ or } gAg^{-1} \geq K
    \end{align*}
    which has the following properties:
    \begin{enumerate}
        \item The concept lattice $\conlat(X,Y,R)$ is canonically isomorphic to $\Tr(\Sub(G))$.
        \item $X \cong J(\Tr(L))$, with the bijection given by $(H,K) \mapsto \floor{H \to K}$.
        \item $Y \cong M(\Tr(L))$, with the bijection given by $(H,K) \mapsto \ceil{H \to K}^\boxslash$.
        \item $R$ is exactly the relation described in \Cref{thm:relation}.
    \end{enumerate}
\end{summary}

The upshot of this discussion is that it is possible to write down $(X,Y,R)$ as a binary matrix for large groups without needing to compute any one transfer system. We can then use software written for FCA to compute the cardinality of $\conlat(X,Y,R)$, whence the number of transfer systems for $G$ (see \cref{app:A} for more details):

\begin{example}
Let $G=S_4$. Then $(X,Y,R)$ is the following binary matrix:
\begin{center}
\begin{adjustbox}{max width=\linewidth/2}
$\begin{bmatrix}
0&1&1&1&1&1&1&1&1&1&1&1&1&1&1&1&1&1&1&1&1&1&1&1&1&1&1&1&1&1&1&1&1&1\\
1&0&1&1&1&1&1&1&1&1&1&1&1&1&1&1&1&1&1&1&1&1&1&1&1&1&1&1&1&1&1&1&1&1\\
1&1&0&1&1&1&1&1&1&1&1&1&1&1&1&1&1&1&1&1&1&1&1&1&1&1&1&1&1&1&1&1&1&1\\
0&0&1&0&1&1&1&1&1&1&1&1&1&1&1&1&1&1&1&1&1&1&1&1&1&1&1&1&1&1&1&1&1&1\\
0&1&0&1&0&1&1&1&1&1&1&1&1&1&1&1&1&1&1&1&1&1&1&1&1&1&1&1&1&1&1&1&1&1\\
1&0&0&1&1&0&1&0&1&1&1&1&1&1&1&1&1&1&1&1&1&1&1&1&1&1&1&1&1&1&1&1&1&1\\
1&1&0&1&1&1&0&1&1&1&1&1&1&1&1&1&1&1&1&1&1&1&1&1&1&1&1&1&1&1&1&1&1&1\\
1&1&0&1&1&1&1&0&1&1&1&1&1&1&1&1&1&1&1&1&1&1&1&1&1&1&1&1&1&1&1&1&1&1\\
0&0&0&0&0&0&0&0&0&0&1&1&1&1&1&1&1&1&1&1&1&1&1&1&1&1&1&1&1&1&1&1&1&1\\
0&1&0&1&0&1&0&0&1&0&1&1&1&1&1&1&1&1&1&1&1&1&1&1&1&1&1&1&1&1&1&1&1&1\\
0&0&1&0&1&1&1&1&1&1&0&1&1&1&1&1&1&1&1&1&1&1&1&1&1&1&1&1&1&1&1&1&1&1\\
0&1&0&1&0&1&1&1&1&1&1&0&1&1&1&1&1&1&1&1&1&1&1&1&1&1&1&1&1&1&1&1&1&1\\
0&0&0&0&0&0&0&0&0&0&0&0&0&0&1&1&1&1&1&1&1&1&1&1&1&1&1&1&1&1&1&1&1&1\\
0&1&0&1&0&1&0&0&1&0&1&0&1&0&1&1&1&1&1&1&1&1&1&1&1&1&1&1&1&1&1&1&1&1\\
0&1&1&0&1&1&1&1&1&1&1&1&1&1&0&1&1&1&1&1&1&1&1&1&1&1&1&1&1&1&1&1&1&1\\
1&0&0&1&1&0&1&0&1&1&1&1&1&1&1&0&1&1&1&1&1&1&1&1&1&1&1&1&1&1&1&1&1&1\\
0&0&0&0&0&0&0&0&0&0&1&1&1&1&0&0&0&1&1&1&1&1&1&1&1&1&1&1&1&1&1&1&1&1\\
0&0&0&0&0&0&0&0&0&0&0&0&0&0&1&0&0&0&1&1&1&1&1&1&1&1&1&1&1&1&1&1&1&1\\
0&1&1&1&0&1&1&1&1&1&1&1&1&1&1&1&1&1&0&1&1&1&1&1&1&1&1&1&1&1&1&1&1&1\\
1&1&1&1&1&1&0&1&1&1&1&1&1&1&1&1&1&1&1&0&1&1&1&1&1&1&1&1&1&1&1&1&1&1\\
1&1&0&1&1&1&1&0&1&1&1&1&1&1&1&1&1&1&1&1&0&1&1&1&1&1&1&1&1&1&1&1&1&1\\
0&1&0&1&0&1&0&0&1&0&1&1&1&1&1&1&1&1&0&0&0&0&1&1&1&1&1&1&1&1&1&1&1&1\\
0&1&0&1&0&1&0&0&1&0&1&1&1&1&1&1&1&1&1&1&0&1&0&1&1&1&1&1&1&1&1&1&1&1\\
1&0&0&1&1&0&1&0&1&1&1&1&1&1&1&1&1&1&1&1&0&1&1&0&1&1&1&1&1&1&1&1&1&1\\
0&0&0&0&0&0&0&0&0&0&1&1&1&1&1&1&1&1&0&0&0&0&0&0&0&1&1&1&1&1&1&1&1&1\\
1&1&0&1&1&1&0&0&1&0&1&1&1&0&1&1&1&1&1&0&0&0&1&1&1&0&1&1&1&1&1&1&1&1\\
0&0&0&0&0&0&0&0&0&0&0&1&0&0&1&1&1&1&1&0&0&0&0&0&0&0&0&1&1&1&1&1&1&1\\
0&1&0&1&0&1&1&0&1&0&1&1&1&1&1&1&1&1&0&1&0&0&1&1&1&1&1&0&1&1&1&1&1&1\\
0&0&0&0&0&0&0&0&0&0&1&1&1&1&1&1&1&1&0&1&0&0&0&0&0&1&1&0&0&1&1&1&1&1\\
0&1&1&1&0&1&0&1&1&0&1&1&1&1&1&1&1&1&0&0&1&0&0&1&1&1&1&1&1&0&1&1&1&1\\
1&0&1&1&1&0&1&1&1&1&1&1&1&1&1&1&1&1&1&1&1&1&1&0&1&1&1&1&1&1&0&1&1&1\\
0&0&1&0&0&0&0&1&0&0&1&1&1&1&1&1&1&1&0&0&1&0&0&0&0&1&1&1&1&0&0&0&1&1\\
0&0&1&0&0&0&0&1&0&0&0&1&0&1&1&1&1&1&0&0&1&0&0&0&0&1&0&1&0&0&0&0&0&1\\
0&1&1&0&0&1&0&1&0&0&1&1&1&1&0&1&0&1&0&0&1&0&0&1&0&1&1&1&1&0&1&0&1&0
\end{bmatrix}$
\end{adjustbox}
\end{center}

From this we can compute that $|\Tr(\Sub(S_4))| = |\conlat(X,Y,R)| =  8961$.
\end{example}

\begin{example}
    $|\Tr(\Sub(A_6))| = 37,799,146,070$ and $(X,Y,R)$ is a $109 \times 109$ binary matrix.
\end{example}

\begin{remark}\label{rem:trim}
    We see that for $\Tr(L)$ the set of join-irreducible and the set of meet-irreducible elements have the same cardinality. In fact, in \cite[Theorem 4.7]{luo2024latticeweakfactorizationsystems} it is proved that $\Tr(L)$ is moreover a \emph{trim} lattice in the sense of \cite{thomas2006analogue} which is a slightly stronger statement.
\end{remark}

\section{Bounds for Concept Lattices}\label{sec:complex}

In this section we will pivot and discuss various results which bound the size of $\conlat(X,Y,R)$ in terms of invariants of  $(X,Y,R)$. Our main reference for this section comes from \cite{albano2017polynomial}.

We begin with a negative result. In \cite{kuznetsov2001lattice} Kuznetsov proves that the problem of determining $|\conlat(X,Y,R)|$ from $(X,Y,R)$ is \#P-complete. In particular, if such a polynomial time algorithm existed then P=NP. Consequently we see that determining $|\Tr(\Sub(G))|$ for an arbitrary group $G$ is out of reach without significant input from group theory. Instead, we shall seek upper bounds for $|\Tr(\Sub(G))|$ that do not rely on such algorithms. While \cite{albano2017polynomial} discusses many different bounds, we will focus only on those which seem most relevant to our setting.

First we have the trivial upper bound: If $(X,Y,R)$ is some formal context then it is clear that
\[
|\conlat(X,Y,R)| \leq 2^{\min\{|X|,|Y|\}}.
\]
Our next bound comes from considering the \emph{contranomial scale}. Recall that the lattice $[1]^k$ is exactly the Boolean lattice with $k$ atoms. The following definition extracts a formal context for this lattice.

\begin{definition}
    Let $S$ be a set. The context $(S,S,\neq)$ is called a \emph{contranomial scale}. We denote by $\mathbb{N}^c(k)$ the contranomial scale with $k$ objects and $k$ attributes, that is, $\mathbb{N}^c(k) = ([k-1],[k-1], \neq)$.
\end{definition}

Say that a context $(X,Y,R)$ is \emph{$\mathbb{N}^c(k)$-free} if there does not exist a subcontext of it order isomorphic to $\mathbb{N}^c(k)$. In terms of the matrix representation of a context, this is equivalent to finding the larger $k$ such that there is no $k \times k$ submatrix which has 0 on the diagonal and 1's elsewhere (where we allow arbitrary reordering of rows and columns).

\begin{theorem}[{\cite[Theorem 8]{albano2017why}}]
    Let $(X,Y,R)$ be a finite $\mathbb{N}^c(k)$-free context. Then
    \[
    |\conlat(X,Y,R)| \leq \sum_{i=0}^{k-1} \binom{|X|}{i}.
    \]
    In particular, if $k \leq \frac{|X|}{2}$ then
    \[
     |\conlat(X,Y,R)| \leq \dfrac{k|X|^{k-1}}{(k-1)!}.
    \]
\end{theorem}

In general, determining the $\mathbb{N}^c(k)$-freeness of a context is an NP-hard problem; in practice it is usually best to use an integer linear programming solver or similar tool. Our interest in this bound, however, comes from the relation to a recently introduced invariant of groups in the setting of homotopical combinatorics in \cite{basis_paper}. Explicitly, any transfer system for $G$ has a well-defined number assigned to it given by the size of a \emph{minimal basis} for that transfer system. The \emph{complexity} of a group, $\mathfrak{c}(G)$, is then largest among these basis sizes. This invariant is exactly the \emph{breadth} of the lattice $\Tr(\Sub(G))$ in the sense of Birkhoff \cite{birkhoff1973lattice}, that is, the maximum number of elements in an irredundant join representation of an element. Using \cite[Corollary 5.3.2]{albano2017polynomial} we obtain the following:

\begin{theorem}\label{thm:complexity}
    Let $G$ be a finite group. Then the following numbers are equal:
    \begin{enumerate}
        \item The complexity of $G$.
        \item The breadth of $\Tr(\Sub(G))$.
        \item The largest $k$ such that the associated context $(X,Y,R)$ is not $\mathbb{N}^c(k)$-free.
    \end{enumerate}
    Hence if $\mathfrak{c}(G) = k$ then
    \[
    |\Tr(\Sub(G))| \leq \sum_{i=0}^{k-1} \binom{|\operatorname{Rel}^*(\Sub(G))/\!/G|}{i}.
    \]
\end{theorem}

In \cite{basis_paper} the complexity is computed for several families of groups, but in general, this seems to be a very difficult problem. As such, for the rest of this section, and paper, we will focus on a far more tangible and computable upper bound for the number of concepts arising from a formal context.

\begin{definition}
    Let $L$ be a finite lattice with at least two elements. Denote by $\delta(L)$ the \emph{density} of the reduced formal context $(J(L), M(L), \leq)$. That is,
    \[\delta(L) = \frac{\lvert \{(A,B) \in J(L) \times M(L) : A \leq B\} \rvert}{\lvert J(L) \rvert \lvert M(L) \rvert}.\]
    Denote by $\rho(L)$ the \emph{codensity} of the reduced formal context $(J(L), M(L), \leq)$. That is,
     \[\rho(L) = \frac{\lvert \{(A,B) \in J(L) \times M(L) : A \not\leq B\} \rvert}{\lvert J(L) \rvert \lvert M(L) \rvert} = 1 - \delta(L).\]
\end{definition}

In other words, $\delta(L)$ is the proportion of entries in the context matrix which are $1$'s while $\rho(L)$ is the proportion of entries which are 0's.

\begin{proposition}[\cite{schuett1987abschaetzungen}]\label{thm:Ibound}
Let $L$ be a finite lattice with associated reduced formal context $(J(L), M(L), \leq)$. Then
\[
|L| \leq \dfrac{3}{2} \left( 2^{\sqrt{\delta(L) \cdot |J(L)| \cdot |M(L)| + 1}} \right) - 1
\]
We note that the factor $\delta(L) \cdot |J(L)| \cdot |M(L)|$ is simply the number of 1's in the matrix representation of the reduced context. 
\end{proposition}

Given this bound, we will now proceed in the rest of the paper to explicitly compute $\delta(\Tr(\Sub(G)))$ for many families of abelian groups, which then provides a bound for $|\Tr(\Sub(G))|$.

\begin{remark}
    In the usual FCA literature where the methods are applied to real-world data applications, one notices that the densities of the contexts are relatively low. We shall see the opposite here, in that the contexts arising from transfer system lattices more often than not have density close to 1.
\end{remark}

\section{The Reduced Context for Cyclic Groups of Prime Power Order}\label{sec:tamari}

In this section we will discuss the formal context matrix for $\Tr(\Sub(C_{p^n}))$ where $p$ is prime. Given the results of \cite{bbr}, this is akin to forming the reduced formal context for the $(n+1)$-Tamari lattice, something that as-of-yet seems to be unexplored in the FCA literature. This calculation will serve as a warm-up for \cref{sec:cyclic} where we compute the (co)density of $\Tr(\Sub(G))$ for arbitrary cyclic groups. Throughout this section we will identify $\Sub(C_{p^n})$ with the total order $[n]$.

We begin with a simple observation regarding the number of join and meet-irreducible elements.

\begin{lemma}\label{lem:nummeettamari}
    Let $n \geq 1$ and $p$ prime. Then 
    \[
    |J(\Tr([n]))| = \binom{n+1}{2} = |M(\Tr([n]))|.
    \]
\end{lemma}
\begin{proof}
    Join and meet-irreducible elements of $\Tr([n])$ correspond bijectively with pairs $(x,y)$ such that $0 \leq x < y \leq n$.
\end{proof}

\begin{theorem}\label{thm: density of tamari}
    Let $n \geq 1$ and $p$ prime. Then
    \[
    \rho(\Tr([n])) =\dfrac{\binom{n+3}{4}}{\binom{n+1}{2}^2} = \frac{(n+2)(n+3)}{6n(n+1)}.
    \]
    In particular,
    \[
    \lim_{n \to \infty} \rho(\Tr([n])) = \dfrac{1}{6}.
    \]
\end{theorem}

\begin{proof}
    We need to compute the number of 0's of the context matrix. To do so, we appeal to \cref{thm:relation} which, using the fact that $[n]$ is totally ordered, asks us to consider tuples $0 \leq k,h,a, b \leq n$ such that
    \[
    k \leq a < h \leq b.
    \]
    The number of such tuples is given by $\binom{n+3}{4}$. Indeed, let $X_0 = \{(k,a,h,b) \in \{0, \dots, n\}^4 : k \leq a < h \leq b\}$, and let $Y_0$ be the set of $4$-element subsets of $\{0, \dots, n\} \amalg \{*_1, *_2\}$. We define a function $\Phi : Y_0 \to X_0$ as follows:
    \begin{align*}
        \Phi(\{a,b,c,d\}) &= (a,b,c,d) && 0 \leq a < b < c < d \leq n \\
        \Phi(\{a,b,c\}\amalg\{*_1\}) &= (a,a,b,c) && 0 \leq a < b < c \leq n \\
        \Phi(\{a,b,c\}\amalg\{*_2\}) &= (a,b,c,c) && 0 \leq a < b < c \leq n \\
        \Phi(\{a,b\}\amalg\{*_1,*_2\}) &= (a,a,b,b) && 0 \leq a < b \leq n
    \end{align*}
    It is straightforward to check that $\Phi$ is bijective.

    The claimed codensity result then follows in combination with \cref{lem:nummeettamari}.
\end{proof}

\begin{remark}
By plotting the context matrix with black and white pixels as discussed in \cref{ex:exampleFCA}, one observes a highly structured situation (\cref{fig:C_p^21 context}). This particular fractal structure should not be surprising given the relation of $\Tr([n])$ to the $(n+1)$-Tamari lattice, and the iterative constructions that exist of the latter.
    \begin{figure}[h!]
    \centering
    \includegraphics[width=0.3\linewidth,cframe=nice-pink 2pt]{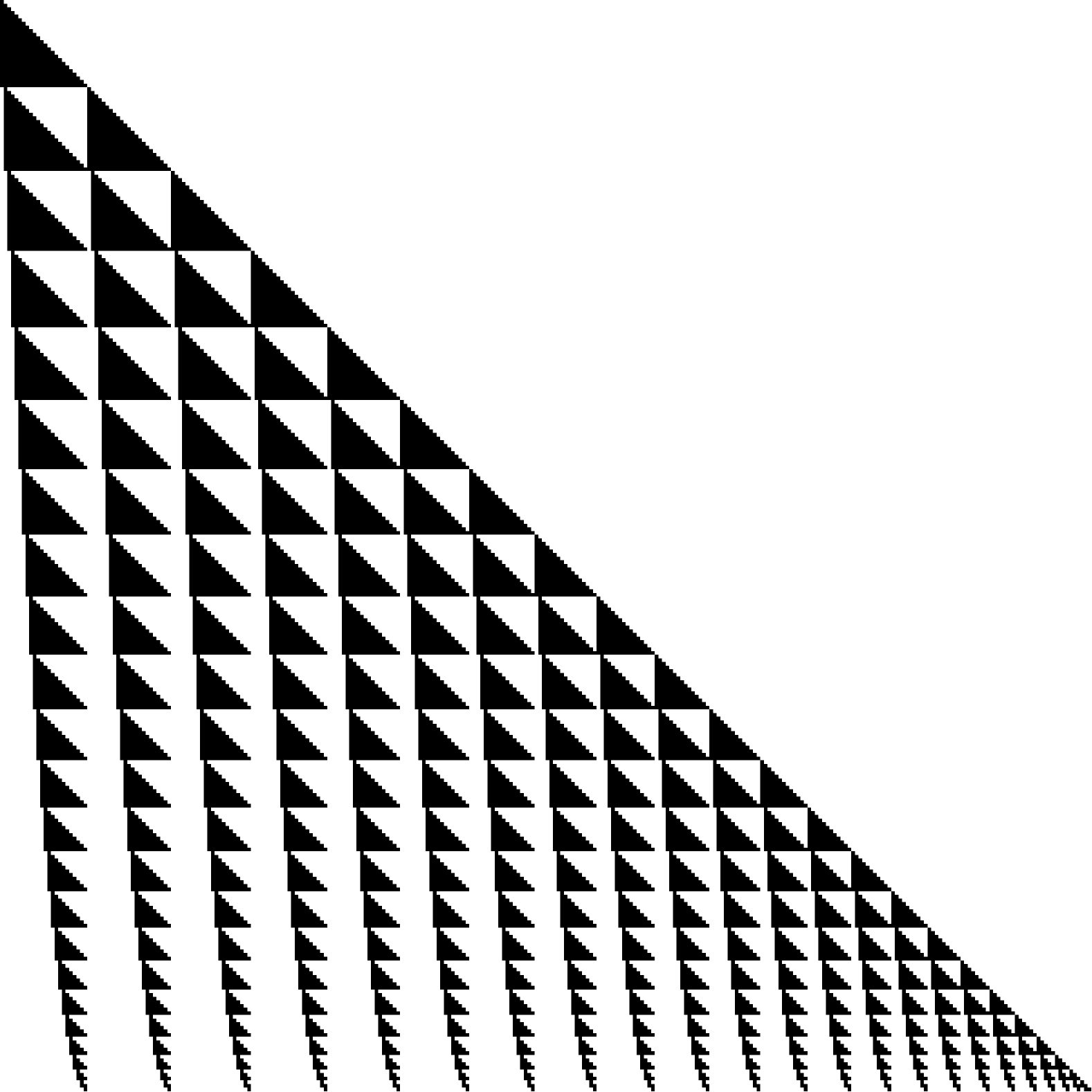}
    \caption{The reduced concept matrix for $\Tr(\Sub(C_{p^{21}})) \cong\mathrm{Tamari}({22})$.}
    \label{fig:C_p^21 context}
\end{figure}
\end{remark}

\section{The Reduced Context for Arbitrary Cyclic Groups}\label{sec:cyclic}

Now that we have addressed the situation for cyclic groups of prime power order, we will generalize this to all cyclic groups, beginning in \cref{subsec:twoprime} with the case of two prime factors to illustrate the main idea. For any positive integer $N$ with prime factorization $p_1^{n_1} \dots p_k^{n_k}$, we recall that there is an isomorphism
\[\Sub(C_N) \cong \Sub(C_{p_1^{n_1}}) \times \dots \times \Sub(C_{p_k^{n_k}}) \cong [n_1] \times \dots \times [n_k].\]

\subsection{Cyclic Groups with Two Prime Factors}\label{subsec:twoprime}

\begin{proposition}
    \label{prop:[n] x [m] codensity}
    Let $n,m \geq 1$ and $p \neq q$. Then
    \[\rho(\Tr(\Sub(C_{p^n q^m})) = \frac{(m+2) (m+3) (n+2) (n+3) (3 m n+4 m+4 n)}{36 (m+1) (n+1) (2m + 2n + m n)^2}.\]
\end{proposition}
\begin{proof}
By \cref{thm:relation} we have a $0$ in the reduced formal context for $\Tr([n] \times [m])$ for each quadruple $(A,B,K,H)$ of points in $[n] \times [m]$ such that $A < B$, $K < H$, $A \geq K$, $B \geq H$, and $A \ngeq H$. We will count the number of such quadruples by the following procedure, where we refer to \Cref{fig:illustration for [n] x [m]}:
    \begin{description}
        \item[Step 1] Choose $K$ to be any point $(i,j) \in [n] \times [m]$.
        \item[Step 2] Choose $H$ to be any point $(s,t) \in [n] \times [m]$ such that $s \geq i$, $t \geq j$, and $(s,t) \neq (i,j)$. In other words, pick $H$ to be any point in above the line $L_1$ besides $K$.
        \item[Step 3] Choose $B$ to be any point $(x,y) \in [n] \times [m]$ such that $x \geq s$ and $y \geq t$. In other words, choose $B$ to be any point above the line $L_2$.
        \item[Step 4] Choose $A$ to be any point in the yellow shaded region.
    \end{description}
        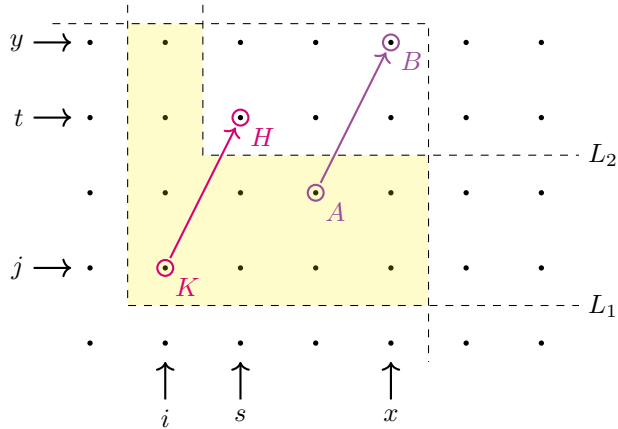
\begin{figure}[h]
    \centering
    \begin{tikzpicture}[scale=1.0]
        \foreach \x in {0,...,6} {
            \foreach \y in {0,...,4} {
                \fill [black] (\x,\y) circle (1pt);
            };
        };
        \fill [yellow, opacity=0.2] (0.5,4.25) -- (1.5,4.25) -- (1.5,2.5) -- (4.5,2.5) -- (4.5,0.5) -- (0.5,0.5) -- cycle;
        \draw [nice-pink, thick] (1,1) circle (3pt) node (K) {} node [below right] {$K$};
        \draw [nice-pink, thick] (2,3) circle (3pt) node (H) {} node [below right] {$H$};
        \draw [nice-pink, thick, ->] (K) -- (H);
        \draw [black, dashed] (0.5,4.5) -- (0.5,0.5) -- (6.5,0.5) node [right] {$L_1$};
        \draw [black, dashed] (1.5,4.5) -- (1.5,2.5) -- (6.5,2.5) node [right] {$L_2$};
        \draw [nice-purple, thick] (4,4) circle (3pt) node (B) {} node [below right] {$B$};
        \draw [black, dashed] (-0.5,4.25) -- (4.5,4.25) -- (4.5,-0.25);
        \draw [nice-purple, thick] (3,2) circle (3pt) node (A) {} node [below right] {$A$};
        \draw [nice-purple, thick, ->] (A) -- (B);
        \draw [black, thick, ->] (1,-0.75) node [below] {$i$} -- (1,-0.25);
        \draw [black, thick, ->] (2,-0.75) node [below] {$s$} -- (2,-0.25);
        \draw [black, thick, ->] (4,-0.75) node [below] {$x$} -- (4,-0.25);
        \draw [black, thick, ->] (-0.75,1) node [left] {$j$} -- (-0.25,1);
        \draw [black, thick, ->] (-0.75,3) node [left] {$t$} -- (-0.25,3);
        \draw [black, thick, ->] (-0.75,4) node [left] {$y$} -- (-0.25,4);
    \end{tikzpicture}
    \caption{Illustration for the counting procedure in the proof of \Cref{prop:[n] x [m] codensity}.}
    \label{fig:illustration for [n] x [m]}
    \end{figure}
    
    This gives a count of
    \[\sum_{i=0}^n \sum_{j=0}^m \sum_{\substack{i \leq s \leq n \\ j \leq t \leq m \\ (s,t) \neq (i,j)}} \sum_{x=s}^n \sum_{y=t}^m [(x-i+1)(y-j+1)-(x-s+1)(y-t+1)]\]
    $0$'s in the reduced formal context. We first note that the exclusion of $(s,t) = (i,j)$ in the above summation has no effect (the yellow region will be empty when $H=K$), so the number of $0$'s is also given by
    \begin{align*}
        & \sum_{i=0}^n \sum_{j=0}^m \sum_{s=i}^n \sum_{t=j}^m \sum_{x=s}^n \sum_{y=t}^m [(x-i+1)(y-j+1)-(x-s+1)(y-t+1)] \\
        &= \sum_{i=0}^n \sum_{j=0}^m \sum_{s=i}^n \sum_{t=j}^m \sum_{x=s}^n -\frac{1}{2} (m-t+1) (-2 i j+i m+i t+2 i+2 j x+2 j-m s+s t-2 s-2 t x-2 t) \\
        &= \sum_{i=0}^n \sum_{j=0}^m \sum_{s=i}^n \sum_{t=j}^m \frac{1}{2} (m-t+1) (n-s+1) (2 i j-i m-i t-2 i-j n-j s-2 j+m s+n t+2 s+2 t) \\
        &= \sum_{i=0}^n \sum_{j=0}^m \sum_{s=i}^n -\frac{1}{12} (j-m-2) (j-m-1) (n-s+1) (-4 i j+4 i m+6 i+j n+3 j s+2 j-m n-3 m s-2 m-6 s) \\
        &= \sum_{i=0}^n \sum_{j=0}^m \frac{1}{12} (i-n-2) (i-n-1) (j-m-2) (j-m-1) (i j-i m-i-j n-j+m n+m+n) \\
        &= \sum_{i=0}^n -\frac{1}{144} (m+1) (m+2) (m+3) (i-n-2) (i-n-1) (3 i m+4 i-3 m n-3 m-4 n) \\
        &= \frac{1}{576} (m+1) (m+2) (m+3) (n+1) (n+2) (n+3) (3 m n+4 m+4 n)
    \end{align*}
    by standard summation techniques\footnote{Another way to approach this summation is to note that the expression must be a polynomial of degree at most $4$ in both $n$ and $m$. One can then compute sufficiently many values by brute force and use polynomial interpolation to find the closed form.}.
    
    Next, we count the number of rows (equivalently, the number of columns) in the reduced formal context for $\Tr([n] \times [m])$. This is simply the number of pairs $(A,B)$ of points in $[n] \times [m]$ such that $A < B$, and is thus given by
    \begin{align*}
        |J(\Tr([n] \times [m]))| &= \sum_{i=0}^n \sum_{j=0}^m ((n-i+1)(m-j+1)-1) \\
        &= \sum_{i=0}^n -\frac{1}{2} (m+1) (i m+2 i-m n-m-2 n)\\
        &= \frac{1}{4} (m+1) (n+1) (m n+2 m+2 n).
    \end{align*}
    Thus, we compute
    \begin{align*}
        \rho(\Tr([n] \times [m])) &=  \frac{\frac{1}{576} (m+1) (m+2) (m+3) (n+1) (n+2) (n+3) (3 m n+4 m+4 n)}{\left(\frac{1}{4} (m+1) (n+1) (m n+2 m+2 n)\right)^2} \\[5pt]
        &= \frac{(m+2) (m+3) (n+2) (n+3) (3 m n+4 m+4 n)}{36 (m+1) (n+1) (2m + 2n + m n)^2}
    \end{align*}
    as desired.
\end{proof}

\begin{remark}
    Setting $m=0$ in \Cref{prop:[n] x [m] codensity} gives exactly the result of \Cref{thm: density of tamari}.
\end{remark}

\begin{corollary}\label{cor:1112bound}
    From \Cref{prop:[n] x [m] codensity} one can compute that
    \[\lim_{m \to \infty} \rho(\Tr([n] \times [m])) = \frac{(n+3) (3 n+4)}{36 (n+1) (n+2)}\]
    and
    \[\lim_{n,m \to \infty} \rho(\Tr([n] \times [m])) = \frac{1}{12}.\]
\end{corollary}

\begin{figure}[h!]
    \centering
    \includegraphics[width=0.3\linewidth,cframe=nice-pink 2pt]{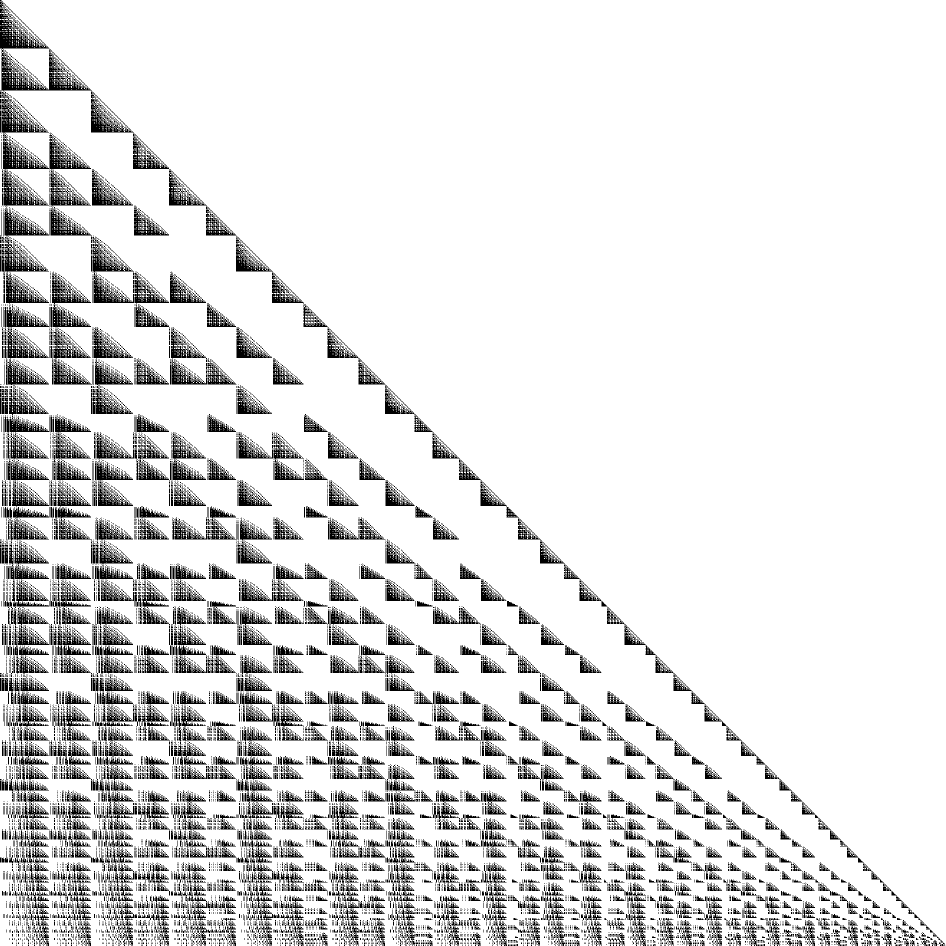}
    \caption{The reduced concept matrix for $\Tr([7] \times [7])$.}
    \label{fig:[7]x[7] context}
\end{figure}

We now can use the calculation of the (co)density to form non-trivial upper bounds on $|\Tr([n] \times [m])|$ using the discussion of \cref{sec:complex}. In \cite{basis_paper} the complexity of $C_{p^nq}$ was computed as
    \[
    \mathfrak{c}(C_{p^nq}) =
    \begin{cases}
        3k+1 &: n=2k \\
        3k+2 &: n=2k+1
    \end{cases}
    \]
    which in conjunction with \cref{thm:complexity} provides an upper bound on the number of transfer systems. 
    
    However, only conjectural evidence was obtained in \cite{basis_paper} for $\mathfrak{c}(C_{p^nq^m})$ whereas \cref{prop:[n] x [m] codensity} works for all $n,m \geq 1$. Using the proof of \cref{prop:[n] x [m] codensity} we see that the trivial upper bound is given by
    \[
    |\Tr([n] \times [m])| \leq 2^{\frac{1}{4} (m+1) (n+1) (m n+2 m+2 n)}.
    \]
    We can combine this with the limiting behavior of the (co)density along with \cref{thm:Ibound} to improve this trivial bound by a factor of $\sqrt{\frac{11}{12}}$ in the exponent:

\begin{corollary}
    For $n, m \geq 1$,
    \[\lvert \Tr([n] \times [m]) \rvert \leq \frac{3}{2}(2^{\frac{\sqrt{33}}{24} (m+1) (n+1) (m n+2 m+2 n)}) -1.\]
\end{corollary}
\begin{proof}
    Let $j = \abs{J(\Tr([n] \times [m]))}$ and let $d = \delta(\Tr([n] \times [m]))$. From the formulas already established, one can show that
    \[dj^2 + 1 \leq \frac{11}{12}j^2\]
    for $n,m \geq 1$, and the result now follows from \cref{thm:Ibound}.
\end{proof}

\subsection{Arbitrary Cyclic Groups}

We now extend the ideas from the case with two prime factors to arbitrarily many prime factors. We can use essentially the same counting procedures to find the relevant terms:

\begin{theorem}\label{thm: codensity arbitrary cyclic}
Let $N = p_1^{n_1} \dots p_k^{n_k}$. Then
\[\rho(\Tr(\Sub(C_N))) = \frac{\sum_{0 \leq x_1 \leq y_1 \leq z_1 \leq n_1} \dots \sum_{0 \leq x_k \leq y_k \leq z_k \leq n_k} \left(\prod_{i=1}^k (z_i - x_i +1) - \prod_{i=1}^k (z_i - y_i +1)\right)}{\left(\sum_{0 \leq x_1 \leq n_1} \dots \sum_{0 \leq x_k \leq n_k} \left(\prod_{i=1}^k (n_i - x_i + 1) - 1 \right)\right)^2}.\]
\end{theorem}

For a fixed value of $k$, one can then use straightforward (albeit tedious) summation techniques to produce a closed form for the codensity in terms of the $n_1, \dots, n_k$. While these calculations can be performed by hand, it is preferable to use software such as Mathematica to undertake the relevant manipulations\footnote{We provide code for this computation in \cite{FCA-Homotopical-Combinatorics_2025}.}.

\begin{corollary}
Let $N = p_1^{n_1}p_2^{n_2}p_3^{n_3}$. Then
\[\adjustbox{width=\linewidth}{$\rho(\Tr(\Sub(C_N))) = 
    \frac{\left(n_1+2\right) \left(n_1+3\right) \left(n_2+2\right) \left(n_2+3\right) \left(n_3+2\right) \left(n_3+3\right) \left(12 n_2 n_1+7 n_2 n_3 n_1+12 n_3 n_1+16 n_1+16 n_2+12 n_2 n_3+16 n_3\right)}{216 \left(n_1+1\right) \left(n_2+1\right) \left(n_3+1\right) \left(2 n_2 n_1+n_2 n_3 n_1+2 n_3 n_1+4 n_1+4 n_2+2 n_2 n_3+4 n_3\right)^2}
$}\]
and hence
\[\lim_{n_1,n_2,n_3 \to \infty} \rho(\Tr(\Sub(C_N))) = \frac{7}{216}.\]
\end{corollary}

One can perform computations for small $k$ to establish the first terms of the limiting codensities:
\renewcommand*{\arraystretch}{1.5}
\begin{table}[h]
\begin{tabular}{|c|c|c|c|c|c|}
\hline
$k$                                                                            & 1             & 2              & 3               & 4               & 5                 \\ \hline
$\lim_{n_1, \dots, n_k \to \infty} \rho(\Tr([n_1] \times \dots \times [n_k]))$ & $\frac{1}{6}$ & $\frac{1}{12}$ & $\frac{7}{216}$ & $\frac{5}{432}$ & $\frac{31}{7776}$ \\ \hline
\end{tabular}
\caption{A table with the first few terms of limiting codensities.}
\end{table}

This sequence has an emerging pattern which we leave as a conjecture.
\begin{conjecture}\label{conjecture: limiting densities cyclic}
    For all positive integers $k$,
    \[\lim_{n_1, \dots, n_k \to \infty} \rho(\Tr([n_1] \times \dots [n_k])) = \frac{2^k-1}{6^k}.\]
\end{conjecture}

\subsection{Squarefree Cyclic Groups}

Among cyclic groups, of special interest to homotopical combinatorics are cyclic groups of squarefree order, since the subgroup lattices are exactly the Boolean lattices $[1]^k$. As such, we will conclude this section by explicitly unwrapping the result of \cref{thm: codensity arbitrary cyclic} in this setting.

\begin{proposition}
Let $k \geq 1$. Then
    \[\rho(\Tr([1]^k)) = \frac{6^k - 5^k}{(3^k-2^k)^2}.\]
\end{proposition}
\begin{proof}
    For arbitrary $k$, setting $n_1 = \dots = n_k = 1$ yields
    \[
        \sum_{0 \leq x_1 \leq y_1 \leq z_1 \leq 1} \cdots \sum_{0 \leq x_k \leq y_k \leq z_k \leq 1} \left(\prod_{i=1}^k (z_i - x_i +1) - \prod_{i=1}^k (z_i - y_i +1)\right)
    \]
    $0$'s in the formal context. We can view this sum as indexed over triples $(x,y,z)$ of $k$-bit binary integers $x = x_1 x_2 \dots x_k,\, y = y_1 y_2 \dots y_k,\, z = z_1 z_2 \dots z_k$ such that
    \[x \preceq y \preceq z \preceq 2^k-1\]
    where $\preceq$ denotes \emph{bitwise} comparison.
    
    Now we will proceed by analyzing each of the two product terms in the above sum in turn. We begin by recording the four possible options for triples of bits $x_i \leq y_i \leq z_i$ in \cref{tab:bits}.
    \renewcommand*{\arraystretch}{1.0}
    \begin{table}[h]
\begin{tabular}{|c|c|c|c|c|}
\hline
Option & 1 & 2 & 3 & 4 \\ \hline
$z_i$ & 0 & 1 & 1 & 1 \\ \hline
$y_i$ & 0 & 0 & 1 & 1 \\ \hline
$x_i$ & 0 & 0 & 0 & 1 \\ \hline
$z_i - x_i$ & 0 & 1 & 1 & 0 \\ \hline
$z_i - y_i$ & 0 & 1 & 0 & 0 \\ \hline
\end{tabular}
\caption{Possible options for bits $x_i \leq y_i \leq z_i$.}\label{tab:bits}
\end{table}

    We have
    \[\sum_{x \preceq y \preceq z \preceq 2^k-1} \prod_{i=1}^k (z_i - x_i +1) = \sum_{x \preceq y \preceq z \preceq 2^k-1} 2^{\sum_i(z_i - x_i)} = \sum_{s=0}^k \sum_{\substack{x \preceq y \preceq z \preceq 2^k-1 \\ \sum_i (z_i - x_i) = s}} 2^s.\]
    Appealing to \cref{tab:bits} we see that to produce a triple $x \preceq y \preceq z \preceq 2^k-1$ such that $\sum_i (z_i - x_i) = s$, we must choose $s$ bits where either $(x_i=0,y_i=0,z_i=1)$ or $(x_i=0,y_i=1,z_i=1)$ and $k-s$ bits where either $(x_i=0,y_i=0,z_i=0)$ or $(x_i=1,y_i=1,z_i=1)$. There are thus a total of
    \[\binom{k}{s} 2^s 2^{k-s}\]
    such triples $x \preceq y \preceq z \preceq 2^k-1$, and we get
    \[\sum_{x \preceq y \preceq z \preceq 2^k-1} \prod_{i=1}^k (z_i - x_i +1) = \sum_{s=0}^k \binom{k}{s} 2^s 2^{k-s} 2^s = 6^k\]
    by the binomial theorem. Similarly, to produce a triple $x \preceq y \preceq z \preceq 2^k-1$ such that $\sum_i (z_i - y_i) = s$, we must choose $s$ bits to be $(0,0,1)$ and $k-s$ bits to be one of $(0,0,0), (0,1,1), (1,1,1)$. There are $\binom{k}{s} 3^{k-s}$ such triples, and thus
    \[\sum_{x \preceq y \preceq z \preceq 2^k-1} \prod_{i=1}^k (z_i - y_i +1) = \sum_{s=0}^k \binom{k}{s} 3^{k-s} 2^s = 5^k.\]
    Putting this together, we obtain
    \[\sum_{0 \leq x_1 \leq y_1 \leq z_1 \leq 1} \cdots \sum_{0 \leq x_k \leq y_k \leq z_k \leq 1} \left(\prod_{i=1}^k (z_i - x_i +1) - \prod_{i=1}^k (z_i - y_i +1)\right)= 6^k - 5^k.\]

    We similarly obtain
    \[\sum_{x \preceq 2^k-1} \prod_{i=1}^k (1 - x_i + 1) - 2^k = \sum_{x \preceq 2^k-1} 2^{k-\sum_i x_i} - 2^k = 3^k-2^k\]
    rows in the formal context, and the result follows.
\end{proof}

\begin{corollary}
\[
    \lim_{k \to \infty} \rho (\Tr([1]^k)) = 0.
\]
\end{corollary}

    \begin{figure}[h!]
    \centering
    \includegraphics[width=0.3\linewidth,cframe=nice-pink 2pt]{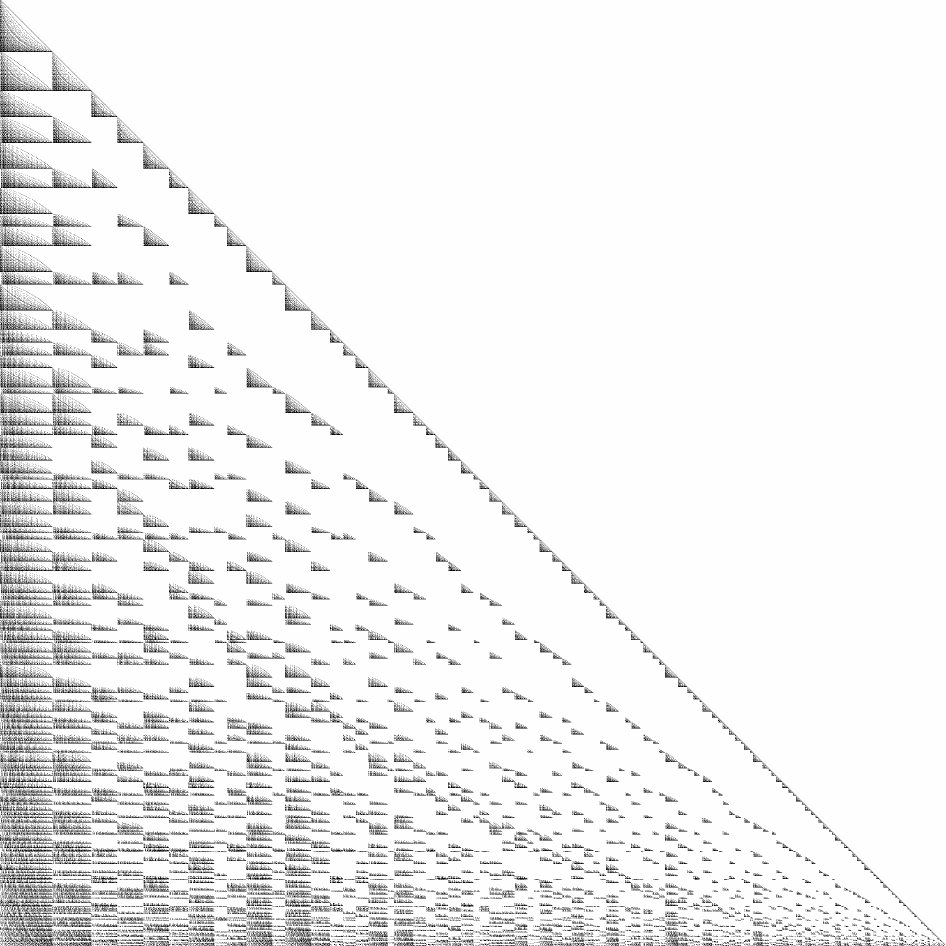}
    \caption{The reduced concept matrix for $\Tr([1]^7)$.}
    \label{fig:cube context}
\end{figure}

\section{The Reduced Context for Elementary Abelian \texorpdfstring{$p$}{p}-Groups}\label{sec:elementary}

In this section we will compute the codensity of the formal context associated to elementary abelian $p$-groups. We will repeatedly view $C_p^n$ as $\mathbb{F}_p^n$. We define the auxiliary function
\[
a_d \coloneq \sum_{i=0}^d \qbinom{d}{i}{p}\]
which counts the total number of subspaces of $\mathbb{F}_p^d$.

\begin{theorem}\label{thm:ele}
Let $n > 0$ and $p$ prime. Then
    \[
\rho(\Tr(\Sub(C_p^n)))  = \dfrac{\sum_{i=0}^{n-1} \sum_{j=1}^{n-i} \sum_{k=0}^{n-i-j} \qbinom{n}{i}{p} \qbinom{n-i}{j}{p} \qbinom{n-i-j}{k}{p} (a_{j+k} - a_k)}{\left( \sum_{i=0}^{n-1} \qbinom{n}{i}{p} (a_{n-i}-1) \right)^2}.
    \]
\end{theorem}

\begin{proof}
We begin by discussing the numerator, that is, the number of 0's in the context. Once again appealing to \cref{thm:relation} we are led to the following steps:
    \begin{description}
        \item[Step 1] Choose $K \leq C_{p}^n$. Suppose that $K$ is of dimension $i$, then we pick a corresponding subspace of $\mathbb{F}_p^n$ of dimension $i$. As we will later pick $H$ with $H > K$, note that the dimension of $K$ cannot be $n$. There are
        \[
        \sum_{i=0}^{n-1} \qbinom{n}{i}{p}
        \]
        such subspaces.
        \item[Step 2] Pick $H > K$. For this we pick the dimension of $H/K$, which we denote $j$, then pick a subspace of $\mathbb{F}_p^n/K$ of dimension $j$. As $H>K$ we know that the dimension of $H$ is $i+j$ and hence $j >0$. This gives us 
        \[
        \sum_{i=0}^{n-1} \qbinom{n}{i}{p} \sum_{j=1}^{n-i} \qbinom{n-i}{j}{p}
        \]
        total combinations of completing this step.
        \item[Step 3] We now pick $B$ with $B \geq H$. As in Step 2, we pick the dimension of $B/H$, say $k$, and then pick a subspace of $\mathbb{F}_p^n/H$ of dimension $k$. Note that $B$ has dimension $i+j+k$. This gives
        \[
        \sum_{i=0}^{n-1} \qbinom{n}{i}{p} \sum_{j=1}^{n-i} \qbinom{n-i}{j}{p} \sum_{k=0}^{n-i-j} \qbinom{n-i-j}{k}{p}
        \]
        total combinations of completing this step.
        \item[Step 4] We are left to pick $A$ with $B \geq A \geq K$ and $A \not\geq H$ (in particular $A \neq B$). To do so we can pick any subspace of $B/K$ and then subtract the subspaces of $B/H$. This gives our final result of
        \[
        \sum_{i=0}^{n-1} \qbinom{n}{i}{p} \sum_{j=1}^{n-i} \qbinom{n-i}{j}{p} \sum_{k=0}^{n-i-j} \qbinom{n-i-j}{k}{p} (a_{j+k}-a_k).
        \]
        Rearranging the terms gives the numerator as claimed.
    \end{description}

    For the denominator we need to count the number of pairs $(K,H)$ with $1\leq K < H \leq C_p^n$. To do so we pick $H$, which cannot be the trivial subgroup, of which there are
    \[
    \sum_{i=0}^{n-1} \qbinom{n}{i}{p}
    \]
    options. We then pick $K$ which is a non-trivial subgroup of $K$, of which gives
    \[
    \sum_{i=0}^{n-1} \qbinom{n}{i}{p} (a_{n-i}-1)
    \]
    options, completing the proof.
\end{proof}
 
\begin{corollary}
Let $n \geq 2$. Then
 \[\lim_{p \to \infty}\rho(\Tr(\Sub(C_p^n))) = \begin{cases} 1/4 &: n = 2 \\ 0 &: \text{else} \end{cases}\]
    and
    \[\lim_{n \to \infty}\rho(\Tr(\Sub(C_p^n))) = 0.\]
for all primes $p$.
\end{corollary}

Recall that $\lim_{p \to 1} \qbinom{n}{k}{p} = \binom{n}{k}$. Using this observation we observe the following relation between the codensity of elementary abelian $p$-groups and the squarefree cyclic groups:

\begin{corollary}
Let $n \geq 1$. Then
\[
\lim_{p \to 1} \rho(\Tr(\Sub(C_p^n))) = \rho(\Tr([1]^n)).
\]
\end{corollary}

\begin{remark}
Again, through tedious summation techniques, or the use of Mathematica one can evaluate the polynomials in \cref{thm:ele} for small values of $k$:
\begin{align*}
\rho(\Tr(\Sub(C_p^2))) &= \dfrac{1 + 3p + p^2}{(2p+1)^2} ,\\
\rho(\Tr(\Sub(C_p^3))) &= \dfrac{15+24p+30p^2+16p^3+6p^4}{(6+6p+6p^2+p^3)^2},\\
\rho(\Tr(\Sub(C_p^4))) &= \dfrac{35+70p+124p^2+150p^3+135p^4+95p^5+46p^6+15p^7+p^8}{(10+12p+17p^2+15p^3+8p^4+3p^5)^2}.
\end{align*}
\end{remark}

We can consider the highest coefficient of the numerator and denominator of the codensity appearing in \cref{thm:ele} to provide an economical approximation of it:

\begin{proposition}
    For fixed $n > 0$, $\rho(\Tr(\Sub(C_p^n)))$ is a rational function in $p$ whose numerator has degree $\floor{n^2/2}$ and whose denominator has degree $2\floor{n^2/3}$.
\end{proposition}
\begin{proof}
    It is well-known that $\qbinom{n}{k}{p}$ is a polynomial of degree $k(n-k)$ in $p$. Consequently, for $n > 0$, we have
    \[\deg_p (a_n-1) = \deg_p \sum_{k=1}^{n} \qbinom{n}{k}{p} = \max_{1 \leq k \leq n} k(n-k) = \floor*{\frac{n}{2}}\ceil*{\frac{n}{2}} = \floor{n^2/4}.\]
    Thus, for $n > 0$, we have
    \begin{multline*}
        \deg_p \left(\sum_{i=0}^{n-1} \qbinom{n}{i}{p} (a_{n-i}-1)\right)^2 = 2\deg_p \sum_{i=0}^{n-1} \qbinom{n}{i}{p} \sum_{k=1}^{n-i} \qbinom{n-i}{k}{p} = 2 \max_{0 \leq i \leq n-1} \left(i(n-i) + \floor*{(n-i)^2/4}\right)\\
        = 2 \floor*{\max_{0 \leq i \leq n-1} (i(n-i) + (n-i)^2/4)} = 2 \floor*{\frac{1}{4} \max_{0 \leq i \leq n-1} (n-i)(n+3i)} \overset{(\star)}{=} 2 \floor*{\frac{1}{4} \floor*{4n^2/3}} = 2 \floor{n^2/3}.
    \end{multline*}
    where $(\star)$ requires some careful case-checking modulo $3$. This is the desired result for the degree of the denominator; for the degree of the numerator we employ the same strategy:
    \begin{multline*}
        \deg_p \sum_{i=0}^{n-1} \sum_{j=1}^{n-i} \sum_{k=0}^{n-i-j} \qbinom{n}{i}{p} \qbinom{n-i}{j}{p} \qbinom{n-i-j}{k}{p} (a_{j+k} - a_k) \\
        = \max_{\substack{0 \leq i \leq n-1 \\ 1 \leq j \leq n-i \\ 0 \leq k \leq n - i - j}} \left(i(n-i) + j(n-i-j) + k(n-i-j-k) + \floor{(j+k)^2/4}\right)
        \overset{(\star\star)}{=} \floor{n^2/2}.
    \end{multline*}
    where $(\star\star)$ requires some careful case-checking modulo $4$. This completes the proof.\qedhere
    %  We will require that the degree of $\qbinom{n}{k}{p}$ is $k(n-k)$. Let us begin by analyzing the denominator of the codensity, where we disregard the squaring for now. A general term of this has the form $f(i,k) = i(n-i) + k(n-i-k)$. We wish to maximize $f(i,k)$ over integers $0 \leq i < n$ and $1 \leq k \leq n-i$. Taking partial derivatives we have
    %  \begin{align*}
    %      \frac{\partial f}{\partial i} &= n-2i-k,\\
    %      \frac{\partial f}{\partial k} &= n-i-2k.
    %  \end{align*}
    % Setting both to zero and solving we obtain $k=i=n/3$. Substituting this into $f(i,k)$ we obtain a global maximum of $n^2/3$ (indeed this is a global maximum given that $f(i,k)$ is quadratic in both variables and concave). As we only want integer values we see that the maximum value is $\floor{n^2/3}$. The squaring of this then gives the factor of 2 as required.

    % For the numerator we employ the same strategy but now for the function
    % \[
    % f(i,j,k,\ell) = i(n-i) + j(n-i-j) + k(n-i-j-k) + \ell(j+k-\ell).
    % \]
    % By taking partial derivatives and solving we obtain a unique global maximum of $(0,n/2,n/2,n/2)$ which when substituted back in provides $n^2/2$, or $\floor{n^2/2}$ when considering only integer values as required.
 \end{proof}

 With more careful analysis one can show that $\rho(\Tr(\Sub(C_p^n))) \in \Theta(p^{-n^2/6})$ (independently in either $p$ or $n$), though we will not do so here.

\begin{remark}
    We record a somewhat amusing link between the above result and integral points on elliptic curves. In \cite{stange2016integral} Stange defines a family of integer sequences $R_n(a,\ell)$ called \emph{elliptic troublemaker sequences}. Our numerator and denominator degrees correspond to $R_n(2,4)$ and $R_n(2,6)$ respectively.
\end{remark}

\begin{remark}
    While between this section and \cref{sec:cyclic} we have computed the (co)density for both families of abelian groups, we are unable to provide a result for an arbitrary abelian group. This is due to the fact that $\Tr(\Sub(G \times H))$ bears little resemblance to $\Tr(\Sub(G))$ or $\Tr(\Sub(H))$ in general. For an explicit example we highlight that $|\Tr(\Sub(C_p))| = 2$ while $|\Tr(\Sub(C_p \times C_p))| = 2^{p+2} + p + 1$ by \cite{bao2023transfersystemsrankelementary}.
\end{remark}

\section{(Co)saturated Transfer Systems}\label{sec:sat}

We will conclude this paper with some considerations of particular classes of transfer systems:

\begin{definition}
    Let $L$ be a $G$-lattice, and let $\mathsf{T} \in \Tr(L)$. Then $\mathsf{T}$ is:
    \begin{itemize}
        \item \emph{Saturated} if it satisfies the two-out-of-three property.
        \item \emph{Cosaturated} if it is of the form $\bigvee_{x \in X} \floor{ x \to \top }$ for some subset $X \subseteq L$, where $\top$ denotes the maximum element of $L$.
    \end{itemize}
    We will denote by $\Sat(L)$ (resp., $\coSat(L)$) the set of saturated (resp., cosaturated) transfer systems on $L$.
\end{definition}

The relevance of these subclassses of transfer systems are reflected in the origins of homotopical combinatorics, that is, in equivariant homotopy theory. The saturated transfer systems are related to \emph{linear isometries} operads while the cosaturated transfer systems correspond to \emph{disc-like} operads.

Just as the collection of all transfer systems forms a complete lattice, so do these subclasses. As such, we can consider the formal context associated to them, which we will see is far more technical than the formal context for all transfer systems. We begin by observing that (as suggested by their names) these classes of transfer systems enjoy a duality:

\begin{proposition}[{\cite[Corollary 4.2]{bose2025combinatoricsfactorizationsystemslattices}}]\label{prop:satcosatdual}
    Let $L$ be a $G$-lattice. Then there is a bijection
    \[
    \Sat(L)^\mathrm{op} \cong  \coSat(L^{\mathrm{op}}).
    \]
    In particular, if $L$ is self-dual (e.g., $L = \Sub(G)$ for $G$ an abelian group) then $\abs{\Sat(L)} = \abs{\coSat(L)}$.
\end{proposition}

\begin{warning}
While $\Sat(L)$ is a complete lattice, it is \emph{not} necessarily a sublattice of $\Tr(L)$. To see this, it suffices to see that the join of two saturated transfer systems in $\Tr(L)$ need not be saturated:
\end{warning}

\begin{example}
    Let $p$ and $q$ be distinct primes. Then the two saturated transfer systems $\mathsf{X}, \mathsf{Y} \in \Sat(\Sub(C_{pq}))$ shown below have non-saturated join in $\mathsf{Tr}(\Sub(C_{pq}))$:
    \[
    \mathsf{X} = \begin{tikzcd}
	C_p & C_{pq} \\
	e & C_q
	\arrow[from=1-1, to=1-2]
	\arrow[from=2-1, to=2-2]
\end{tikzcd}
\qquad\qquad
\mathsf{Y} = \begin{tikzcd}
	C_p & C_{pq} \\
	e & C_q
	\arrow[from=2-1, to=1-1]
\end{tikzcd}
\qquad\qquad
\mathsf{X} \vee \mathsf{Y} = \begin{tikzcd}
	C_p & C_{pq} \\
	e & C_q
	\arrow[from=1-1, to=1-2]
	\arrow[from=2-1, to=1-1]
	\arrow[from=2-1, to=1-2]
	\arrow[from=2-1, to=2-2]
\end{tikzcd}
    \]
    The join in $\Sat(\Sub(C_{pq}))$ would instead be computed by taking the above join and then closing up under two-out-of-three. That is, the join in this case would be the complete transfer system.
\end{example}

\begin{warning}
In \cref{rem:trim} we discussed how $\Tr(L)$ is a trim lattice, and thus, in particular the number of join-irreducible elements was equal to the number of meet-irreducible elements. This once again breaks down as soon as we move to (co)saturated transfer systems, where these numbers will differ:
\end{warning}

\begin{example}
    Let $p$ and $q$ be distinct primes. Then one can check that in $\Sat(\Sub(C_{pq}))$ there are four join-irreducible elements, but only three meet-irreducible elements.
\end{example}

Having outlined the differences between all transfer systems and (co)saturated ones insofar as we need to discuss contexts, we move towards some partial results regarding the formal context of $\Sat(L)$ and $\coSat(L)$. We begin with the join-irreducibles of the latter given their simple description.

\begin{lemma}\label{lem:onlysourceG}
    Let $L$ be a $G$-lattice with maximum element $\top$, and let $x \in L$. Then
    \[\{y \in L : y \to \top \in \floor{x \to \top}\} = \left\{\bigwedge_{g \in S} (g \cdot x) : S \subseteq G\right\}.\]
\end{lemma}

\begin{proof}
    This follows from Rubin's description of $\floor{x \to \top}$ in \cite[Proposition A.7]{rubin}:
    \[
    \floor{x \to \top} = \left\{ z \wedge \bigwedge_{g \in S} (g \cdot x) \to z : z \in L, S \subseteq G \right\}.\qedhere
    \]
\end{proof}

\begin{proposition}
    Let $L$ be a $G$-lattice with maximum element $\top$. Then the set of join-irreducible elements of $\coSat(L)$ is equal to the set of transfer systems generated by a single non-identity relation with target $\top$, i.e.,
    \[
    J(\coSat(L)) = \{ \floor{x \to \top} : x \neq \top\}.
    \]
    Moreover, there is a bijection between $J(\coSat(L))$ and the set $(L \setminus \{\top\})/\!/G$. In particular, the number of join-irreducible cosaturated transfer systems on $L$ is equal to $\abs{L/\!/G}-1$.
\end{proposition}

\begin{proof}
    Since every cosaturated transfer system is a join of those of the form $\floor{x \to \top}$ with $x \neq \top$, it is clear that $J(\coSat(L)) \subseteq \{ \floor{x \to \top} : x \neq \top\}$.

    In the other direction, suppose $\floor{x \to \top} = \bigvee \mathcal{T}$ for some $\mathcal{T} \subseteq J(\coSat(L))$, where $x \neq \top$. By \cite[Theorem A.2]{rubin}, we have
    \[x \to \top \in \left(\bigcup_{\mathsf{T} \in \mathcal{T}} \{z \wedge a \to z : a \to b \in \mathsf{T}, z \leq b\}\right)^\circ,\]
    so in particular there is some $\mathsf{T} \in \mathcal{T}$, some $a \to b \in \mathsf{T}$, and some $z \leq b$ such that $z \wedge a = x \neq z$. Since $\mathsf{T} \in J(\coSat(L))$, we have $\mathsf{T} = \floor{y \to \top}$ for some $y \neq \top$. Then since $\mathsf{T} \subseteq \floor{x \to \top}$, we have $y = \bigwedge_{g \in S} (g \cdot x)$ for some $S \subseteq G$. Since $y \neq \top$, $S$ is nonempty, so we may pick some $g \in G$ such that $y \leq g \cdot x$. Furthermore, $a \to b \in \mathsf{T} = \floor{y \to \top}$ yields $a = b \wedge \bigwedge_{g' \in S'} (g' \cdot y)$ for some $S' \subseteq G$. Since $z \wedge a \neq z \leq b$, we must have $a \neq b$, so $S'$ is nonempty. Thus, we may pick some $g' \in G$ such that $a \leq g' \cdot y$. Now we have
    \[x = z \wedge a \leq a \leq g' \cdot y \leq g' \cdot (g \cdot x) = (g'g) \cdot x.\]
    This implies that $x = (g'g) \cdot x$ and thus that $x = g' \cdot y$. We conclude that
    \[\floor{x \to \top} = \floor{g' \cdot y \to \top} = \floor{y \to \top} = \mathsf{T} \in \mathcal{T},\]
    as desired.\qedhere
\end{proof}

Using \cref{prop:satcosatdual} we are afforded an immediate description of the meet-irreducible elements of $\Sat(L)$:

\begin{corollary}\label{cor:meetirrofsat}
    Let $L$ be a $G$-lattice with minimum element $\bot$. Then the set of meet-irreducible elements of $\Sat(L)$ is equal to the set of right lifts of cotransfer systems generated by a single non-identity relation with source $\bot$, i.e.,
    \[
    M(\Sat(L)) = \{ \ceil{\bot \to x}^\boxslash : x \neq \bot\}
    \]
    Moreover, there is a bijection between $M(\Sat(L))$ and the set $(L\setminus\{\bot\})/\!/G$. In particular, the number of meet-irreducible saturated transfer systems on $G$ is equal to $\abs{L/\!/G}-1$.
\end{corollary}

We now focus our attention on the special case $L = \Sub(G)$, since these are the lattices of interest in homotopical combinatorics. While \cref{cor:meetirrofsat} provides a formal description of the meet-irreducible saturated transfer systems here, by unwrapping the lifting (as in the proof of \cref{thm:relation}) we obtain a more straightforward understanding of these transfer systems:
\[
    \ceil{1 \to H}^{\boxslash} = \{A \to B : gHg^{-1} \not\leq B \text{ or } gHg^{-1} \leq A \text{ for all } g\in G\}.
\]
For abelian $G$, this is equivalently $\{A \to B : H \notin \Sub(B) \setminus \Sub(A)\}$.

To complete the story we need to move to the more complex task of identifying the join-irreducible elements of $\Sat(\Sub(G))$ (which by \ref{prop:satcosatdual} will give us a description of the meet-irreducible elements of  $\coSat(\Sub(G))$). The classification of join-irreducible transfer systems will be investigated in future work, with the special case of abelian groups being solved in forthcoming work of the second author and S.\ Bernstein.  We end by discussing some of the technicalities that appear.

Recall that a subgroup $H$ is said to \emph{cover} $K$ if $K < H$ and there is no $L$ with $K < L < H$. We denote by $\Cov(\Sub(G))$ the set $\{K \to H : H \text{ covers } K\}$. 

\begin{proposition}\label{prop:satjoins}
    Let $\mathsf{T}$ be a join-irreducible element of $\Sat(\Sub(G))$. Then $\mathsf{T}$ is generated (as a saturated transfer system) by a single element of $\Cov(\Sub(G))$.
\end{proposition}
\begin{proof}
Let $\mathsf{T} \in J(\Sat(\Sub(G)))$. Rubin \cite[Proposition 5.8]{rubin} shows that $\mathsf{T}$ is generated (as a transfer system) by $\mathsf{T} \cap \Cov(\Sub(G))$. As a consequence, $\mathsf{T}$ is generated as a \emph{saturated} transfer system by $\mathsf{T} \cap \Cov(\Sub(G))$. Thus,
\[\mathsf{T} = \bigvee_{K \to H \in \mathsf{T} \cap \Cov(\Sub(G))} \floor{K \to H}_{\mathrm{sat}},\]
where $\floor{K \to H}_{\mathrm{sat}}$ denotes the saturated transfer system generated by $K \to H$. But $\mathsf{T}$ is join-irreducible!
\end{proof}

The converse surprisingly does not hold. That is, there can be covering relations which do \emph{not} generate join-irreducible saturated transfer systems.

\begin{example}
    Let $G=A_5$. Then $\abs{\Cov(\Sub(G))/\!/G} = 13$. The first thing to note is that the operation of taking a covering relation and forming its saturated closure is not injective. Indeed, the saturated closures of $A_4 \to A_5$, $D_5 \to A_5$, and $S_3 \to A_5$ are all equal (they are all the complete transfer system). This is the only failure of injectivity for $G = A_5$, and it follows that there are exactly 11 saturated transfer systems which are saturated closures of a single cover relation. However, only 8 out of these 11 saturated transfer systems are join-irreducible elements of $\Sat(\Sub(G))$.
\end{example}

\appendix

\section{Notes on Computer-Assisted Computations}\label{app:A}

The ordering of the rows and columns in a formal context has no effect on the corresponding finite lattice. However, this ordering can have an effect on the runtime of certain algorithms which enumerate formal concepts. The computations in this paper were performed with PCbO \cite{pcbo}, for which it seems to be optimal to sort the rows of the context matrix in decreasing order, where each row is viewed as a binary number with its most-significant digit in the last column.\footnote{There are many software packages available for computations in formal concept analysis; we chose to use PCbO for its built-in parallelism.}

PCbO expects as input a file in ``FIMI'' format (these files are conventionally given the file extension \texttt{.dat}). A \texttt{.dat} file has the following format:
\begin{itemize}
    \item Each line of the file is a space-separated list of natural numbers, representing an object in the formal context.
    \item Each natural number in the space-separated list represents an attribute the object has, with \texttt{0} being the first attribute.
\end{itemize}

See \Cref{fig:example dat} for an example.
\begin{figure}[h]
    \begin{minipage}{0.45\textwidth}
    \[\begin{array}{c|c|c|c|c}
           & y_1 & y_2 & y_3 & y_4 \\ \hline
       x_1 & 1   & 0   & 0   & 0   \\ \hline
       x_2 & 0   & 0   & 0   & 1   \\ \hline
       x_3 & 1   & 1   & 1   & 0
    \end{array}\]
    \end{minipage}
    \begin{minipage}{0.2\textwidth}
    \vspace{1em}
    \begin{lstlisting}
0
3
0 1 2
    \end{lstlisting}
    \end{minipage}
    \caption{A formal context and its representation as a \texttt{.dat} file.}
    \label{fig:example dat}
\end{figure}

The published version of the PCbO software prints each computed formal concept to \texttt{stdout}, which slows runtime significantly. Moreover, each thread records the number of formal concepts produced in a C integer, which on most systems has no more than $4$ bytes, and thus may cause overflow errors when more than $2^{31}-1$ concepts are computed. The authors have produced a modified version of PCbO which addresses these limitations, as well as GAP code which generates \texttt{.dat} files encoding the reduced formal context of the lattice of transfer systems on a finite group, available at \url{https://github.com/diracdeltafunk/FCA-Homotopical-Combinatorics}.

The first author has previously published \href{https://github.com/bifibrant/ninfty}{ninfty} \cite{balchin2025ninftysoftwarepackagehomotopical} which allows for more sophisticated computations in homotopical combinatorics. For the basic task of enumerating transfer systems, PCbO is much faster, and uses $O(1)$ space. See \Cref{table:runtimes} for run times in sample computations. This allows for many enumeration computations which previously required the use of high-performance computing clusters to be performed on typical laptop computers.

\begin{table}[h]
    \centering
    \begin{tabular}{r | lll | lll}
        & \multicolumn{3}{c|}{Code from \cite{balchin2025ninftysoftwarepackagehomotopical}} & \multicolumn{3}{c}{Code from \cite{FCA-Homotopical-Combinatorics_2025}} \\[0.3em]
        Lattice & $T_1$ & $T_2$ & $T_1+T_2$ & $T_1$ & $T_2$ & $T_1+T_2$ \\[0.3em] \hline
        &&&&&& \\[-0.9em]
        $\Tr(\Sub(C_{2^{10}}))$ & 3.593s & 4.567s & $\approx$8s & 1.092s & 0.007s & $\approx$1s \\[0.4em]
        $\Tr(\Sub(C_2^3))$ & 2.653s & 638.216s & $\approx$11min & 0.801s & 0.267s & $\approx$1s \\[0.4em]
        $\Tr(\Sub(C_{2^3 \cdot 3^3}))$ & 4.005s & 1132.035s & $\approx$19min & 0.872s & 1.138s & $\approx$2s \\[0.4em]
        $\Tr(\Sub(S_5))$ & 1086.861s & $\infty$ & $\infty$ & 1.275s & 11.478s & $\approx$13s \\[0.4em]
        $\Tr(\Sub(A_6))$ & $>$24 hours & $\infty$ & $\infty$ & 2.471s & 8063.045s & $\approx$2.2 hours \\[0.4em]
    \end{tabular}
    \vspace{1em}
    \caption{Run times for sample computations. $T_1$ denotes the time to produce the necessary input files for the enumeration algorithm (a \texttt{.h} file produced by a sage script or a \texttt{.dat} file propduced by a GAP script). $T_2$ denotes the time to enumerate the elements of the lattice (using either ninfty or PCbO).
    All computations were performed on a 2024 Apple M3 MacBook Air with 16GB of RAM. $\infty$ indicates that the computation is infeasible (the system has insufficient memory, \emph{and} with more memory the computation would take more than 1 month to complete). PCbO was run with 8 threads.}
    \label{table:runtimes}
\end{table}

\bibliography{bib}\bibliographystyle{alpha}

\end{document}